\documentclass[11pt]{amsart}

\newtheorem{thm}{Theorem}[section]
\newtheorem*{thm*}{Theorem}

\newtheorem{cor}[thm]{Corollary}

\newtheorem{lem}[thm]{Lemma}
\newtheorem*{lem*}{Lemma}
\newtheorem{mainthm}{Theorem}
\newtheorem*{mainthm*}{Theorem}
\newtheorem{maincor}[mainthm]{Corollary}
\newtheorem{prop}[thm]{Proposition}

\theoremstyle{definition}

\newtheorem*{case*}{Case}

\newtheorem{defn}[thm]{Definition}
\newtheorem*{defn*}{Definition}

\newtheorem*{exmp*}{Example}

\newtheorem{step}{Step}\renewcommand{\thestep}{}

\theoremstyle{remark}

\renewcommand{\thecase}{}

\newtheorem{rmk}[thm]{Remark}
\newtheorem*{rmk*}{Remark}

\makeatletter
\def\alphenumi{
  \def\theenumi{\alph{enumi}}
  \def\p@enumi{\theenumi}
  \def\labelenumi{(\@alph\c@enumi)}}
\makeatother

\makeatletter
\def\thecase{\@arabic\c@case}
\makeatother

\makeatletter
\def\thestep{\@arabic\c@step}
\makeatother

\newcount\hh
\newcount\mm
\mm=\time
\hh=\time
\divide\hh by 60
\divide\mm by 60
\multiply\mm by 60
\mm=-\mm
\advance\mm by \time
\def\hhmm{\number\hh:\ifnum\mm<10{}0\fi\number\mm}

\let\oldmarginpar\marginpar
\renewcommand\marginpar[1]{\-\oldmarginpar[\raggedleft\footnotesize #1]%
{\raggedright\footnotesize #1}}

\renewcommand\emptyset{\varnothing}

\newcommand\CC{\mathbb{C}}

\newcommand\KK{\mathbb{K}}

\newcommand\PP{\mathbb{P}}

\newcommand\RR{\mathbb{R}}

\newcommand\ZZ{\mathbb{Z}}

\newcommand\cG{{\mathcal{G}}}

\newcommand\fg{{\mathfrak{g}}}

\newcommand\sA{{\mathscr{A}}}
\newcommand\sB{{\mathscr{B}}}
\newcommand\sC{{\mathscr{C}}}

\newcommand\sE{{\mathscr{E}}}

\newcommand\sH{{\mathscr{H}}}

\newcommand\sW{{\mathscr{W}}}

\newcommand\bx{{\mathbf{x}}}

\newcommand{\cov}{\nabla}

\newcommand\eps{\varepsilon}

\newcommand\GL{\operatorname{GL}}
\newcommand\Or{\operatorname{O}}

\newcommand\SO{\operatorname{SO}}

\newcommand\SU{\operatorname{SU}}
\newcommand\U{\operatorname{U}}

\newcommand\less{\setminus}

\newcommand\ad{{\operatorname{ad}}}
\newcommand\Ad{{\operatorname{Ad}}}

\newcommand\Aut{\operatorname{Aut}}

\newcommand\Coker{\operatorname{Coker}}

\newcommand\dist{\operatorname{dist}}

\newcommand\End{\operatorname{End}}

\DeclareMathOperator{\Inj}{Inj}

\newcommand\Ker{\operatorname{Ker}}

\newcommand\Riem{\operatorname{Riem}}

\newcommand\supp{\operatorname{supp}}

\newcommand\Sym{\operatorname{Sym}}

\newcommand\tr{\operatorname{tr}}

\newcommand\vol{\operatorname{vol}}
\newcommand\Vol{\operatorname{Vol}}

\newcommand\afortiori{{\emph{a fortiori }}}

\newcommand\apriori{{\emph{a priori }}}

\newcommand\deRham{{\mathrm{de Rham}}}

\newcommand\loc{{\mathrm{loc}}}

\numberwithin{equation}{section}

\usepackage{amssymb}
\usepackage{graphicx}
\usepackage{hyperref}
\usepackage{mathrsfs}
\usepackage[usenames]{color}
\usepackage{paralist}
\usepackage{slashed}
\usepackage{url}
\usepackage{verbatim}

\hypersetup{pdftitle={Energy gap for Yang-Mills connections, {I}: Four-dimensional closed Riemannian manifolds}}
\hypersetup{pdfauthor={Paul M. N. Feehan}}

\begin{document}

\title[Energy gap for Yang-Mills connections]{Energy gap for Yang-Mills connections, I: Four-dimensional closed Riemannian manifolds}

\author[Paul M. N. Feehan]{Paul M. N. Feehan}
\address{Department of Mathematics, Rutgers, The State University of New Jersey, 110 Frelinghuysen Road, Piscataway, NJ 08854-8019, United States of America}
\email{feehan@math.rutgers.edu}
\curraddr{School of Mathematics, Institute for Advanced Study, Princeton, NJ 08540}
\email{feehan@math.ias.edu}

%COMMENT Remove hours and minutes for arXiv and journal versions and fix date
%\date{\today{ }\hhmm}
\date{This version: April 6, 2016, incorporating final galley proof corrections. \emph{Advances in Mathematics} \textbf{296} (2016), 55--84, doi: 10.1016/j.aim.2016.03.034}

\begin{abstract}
We extend an $L^2$ energy gap result due to Min-Oo \cite[Theorem 2]{Min-Oo_1982} and Parker \cite[Proposition 2.2]{ParkerGauge} for Yang-Mills connections on principal $G$-bundles, $P$, over closed, connected, four-dimensional, oriented, smooth manifolds, $X$, from the case of positive Riemannian metrics to the more general case of good Riemannian metrics, including metrics that are generic and where the topologies of $P$ and $X$ obey certain mild conditions and the compact Lie group, $G$, is $\SU(2)$ or $\SO(3)$.
\end{abstract}

% AMS 2010 subject classifications (used in AMS journals)


% AMS keywords (used in AMS journals)

% Acknowledge support
\thanks{The author is grateful for the hospitality provided by the Department of Mathematics at Columbia University during the preparation of this article. His research is partially supported by National Science Foundation grant DMS-1510064 and the Oswald Veblen Fund and Fund for Mathematics (Institute for Advanced Study, Princeton).}

\maketitle
%TODO Remove for journal version
\tableofcontents

\section{Introduction}
\label{sec:Introduction}
The purpose of this article is to prove that energies associated to non-minimal Yang-Mills connections on a principal bundle, with compact Lie structure group, over a closed, connected, four-dimensional, oriented, Riemannian, smooth manifold are separated from the energy of the minimal Yang-Mills connections by a uniform positive constant depending at most on the Riemannian metric on the base manifold and the Pontrjagin degree of the principal bundle. In particular, rather than require that the Riemannian metric be positive (see Definition \ref{defn:Freed_Uhlenbeck_page_174_positive_metric}) as assumed by Min-Oo in \cite[Theorem 2]{Min-Oo_1982} and Parker in \cite[Proposition 2.2]{ParkerGauge} and constrain the four-dimensional manifold to have negative definite intersection form, we instead assume that the Lie structure group is $\SU(2)$ or $\SO(3)$ and the Riemannian metric is generic in the sense of \cite{FU}, and impose mild conditions on the topology of the principal bundle and four-dimensional manifold inspired by those employed in the most general definitions of the Donaldson invariants \cite{DonPoly,DK,  KMStructure, MorganMrowkaPoly} of the four-dimensional manifold. Our proof leans heavily on analytical results in Yang-Mills gauge theory developed by the author in \cite{Feehan_yang_mills_gradient_flow_v2} and, given those, our main result follows by adapting the method of Min-Oo in \cite{Min-Oo_1982}. Our companion article \cite{Feehan_yangmillsenergygapflat} establishes an $L^{d/2}$ energy gap for Yang-Mills connections over closed Riemannian manifolds of arbitrary dimension $d \geq 2$.

In Section \ref{subsec:Yang-Mills_energy_functional_as_Morse_function}, we review the essentials of gauge theory over four-dimensional manifolds needed to introduce the Yang-Mills energy functional as a Morse function on the quotient space of connections modulo gauge transformations. We recall the $L^2$ energy gap result of Min-Oo \cite[Theorem 2]{Min-Oo_1982} and Parker \cite[Proposition 2.2]{ParkerGauge} for a Yang-Mills connection over a closed, four-dimensional manifold with a positive Riemannian metric in Section \ref{subsec:Yang-Mills_energy_gap}. We state our generalization of their result in Section \ref{subsec:Main_results}, where the requirement that the Riemannian metric be positive (but the Lie structure group, $G$, is arbitrary) is traded for the requirements that the Riemannian metric be generic and $G$ be $\SU(2)$ or $\SO(3)$, together with mild conditions on the topology of $P$ and $X$. Section \ref{subsec:Outline} provides a guide to the remainder of the article.

\subsection{The Yang-Mills energy functional as a Morse function}
\label{subsec:Yang-Mills_energy_functional_as_Morse_function}
Let $G$ be a compact Lie group and $P$ a principal $G$-bundle over a closed, connected, four-dimensional, oriented, smooth manifold, $X$, with Riemannian metric, $g$, and define the associated Yang-Mills energy functional by
\begin{equation}
\label{eq:Yang-Mills_energy_functional}
\sE_g(A)  := \frac{1}{2}\int_X |F_A|^2\,d\vol_g,
\end{equation}
where $A$ is a connection on $P$ of Sobolev class $W^{k,2}$, for an integer $k\geq 1$, and its curvature is denoted by $F_A \in W^{k-1,2}(X; \Lambda^2\otimes\ad P)$. Here, $\Lambda^p := \Lambda^p(T^*X)$ for integers $p\geq 1$ and $\ad P := P\times_{\ad}\fg$ is the real vector bundle associated to $P$ by the adjoint representation of $G$ on its Lie algebra, $\Ad:G \ni u \to \Ad_u \in \Aut\fg$, and fiber metric defined through the Killing form on $\fg$ (see Section \ref{sec:Taubes_1982_Appendix}).

A connection, $A$ on $P$, is a \emph{critical point} of $\sE_g$ --- and by definition a \emph{Yang-Mills connection}  with respect to the metric $g$ --- if and only if it obeys the \emph{Yang-Mills equation} with respect to $g$,
\begin{equation}
\label{eq:Yang-Mills_equation}
d_A^{*,g}F_A = 0 \quad\hbox{a.e. on } X,
\end{equation}
since $d_A^{*,g}F_A = \sE_g'(A)$ when the gradient of $\sE = \sE_g$ is defined by the $L^2$ metric \cite[Section 6.2.1]{DK}, \cite{GroisserParkerSphere} and $d_A^* = d_A^{*,g}: \Omega^p(X; \ad P) \to \Omega^{p-1}(X; \ad P)$ is the $L^2$ adjoint of the exterior covariant derivative \cite[Section 2.1.2]{DK}, $d_A:\Omega^p(X; \ad P) \to \Omega^{p+1}(X; \ad P)$, for integers $p\geq 0$, where $\Omega^p(X; \ad P) = C^\infty (X; \Lambda^p\otimes\ad P)$.

The energy functional, $\sE_g$, is gauge-invariant and thus descends to a function on the quotient space, $\sB(P,g) := \sA(P)/\Aut P$, of the affine space, $\sA(P)$, of connections on $P$ (of Sobolev class $W^{k,2}$ with $k\geq 2$), modulo the action of the group, $\Aut P$, of automorphisms or gauge transformations (of Sobolev class $W^{k+1,2}$) of the principal $G$-bundle, $P$. See \cite[Section 4.2]{DK} or \cite[Chapter 3]{FU} for constructions of the Banach manifold structures on $\sB^*(P,g)$ and $\Aut P$, where $\sB^*(P,g)\subset \sB(P,g)$ denotes the open subset consisting of gauge-equivalence classes of connections on $P$ whose isotropy group is minimal, namely the center of $G$ \cite[p. 132]{DK}. A fundamental investigation of the extent to which $\sE_g$ serves as a Morse function on $\sB(P,g)$ --- despite non-compactness --- has been provided by Taubes in \cite{TauFrame} (see also \cite{TauPath, TauStable} for related results due to Taubes).

The quotient space, $\sB(P,g)$, is non-compact due to its infinite dimensionality. For any $C^\infty$ connection, $A$, we recall that \cite[Equation (2.1.25)]{DK}
\begin{equation}
\label{eq:Donaldson_Kronheimer_2-1-25}
F_A = F_A^{+,g} + F_A^{-,g} \in \Omega^2(X;\ad P) = \Omega^{+,g}(X;\ad P) \oplus \Omega^{-,g}(X;\ad P),
\end{equation}
corresponding to the positive and negative eigenspaces, $\Lambda^{\pm,g}$, of the Hodge star operator $*_g:\Lambda^2 \to \Lambda^2$ defined by the Riemannian metric, $g$, so \cite[Equation (1.3)]{TauSelfDual}
\begin{equation}
\label{eq:Taubes_1982_1-3}
F_A^{\pm,g} = \frac{1}{2}(1 \pm *_g)F_A \in \Omega^{\pm,g}(X; \ad P).
\end{equation}
Of course, similar observations apply more generally to connections, $A$, of Sobolev class $W^{k,2}$.

When there is no ambiguity, we suppress explicit mention of the underlying Riemannian metric, $g$, and write $*=*_g$, and $d_A^* = d_A^{*,g}$, and $d_A^\pm = d_A^{\pm,g}$, and $\sE = \sE_g$, and $F_A^\pm = F_A^{\pm,g}$, and so on.

If the first Pontrjagin numbers of $P$ are non-positive, the finite-dimensional subvariety, $M(P,g) \subset \sB(P,g)$, of gauge-equivalence classes of solutions to the \emph{anti-self-dual} equation with respect to $g$,
\begin{equation}
\label{eq:Moduli_space_anti-self-dual_connections}
M(P,g) := \left\{[A] \in \sB(P,g): F_A^{+,g} = 0 \quad\hbox{a.e. on } X\right\},
\end{equation}
is non-compact (assuming non-empty) due to the phenomenon of energy bubbling characteristic of Yang-Mills gauge theory over four-dimensional manifolds \cite{DK, FU}. As we recall in Section \ref{sec:Taubes_1982_Appendix}, the variety, $M(P,g)$, comprises the set of absolute minima or `ground states' for $\sE_g$.

If the first Pontrjagin numbers of $P$ are non-negative, the finite-dimensional subvariety, $M(P,g) \subset \sB(P,g)$, of gauge-equivalence classes, $[A]$, of solutions to the self-dual equation with respect to the metric $g$, namely $F_A^{-,g} = 0$ a.e. on $X$, again comprises the set of absolute minima for $\sE_g$. By reversing the orientation of $X$, we may restrict our attention without loss of generality to the case where the first Pontrjagin numbers of $P$ are non-positive.

By a result due to Sibner, Sibner, and Uhlenbeck \cite{SibnerSibnerUhlenbeck}, there exist non-minimal critical points of the Yang-Mills energy functional on $P = S^4\times \SU(2)$ and, more generally, principal $\SU(2)$-bundles, $P$, over $S^4$ for any $c_2(P) \geq 2$ by work of Bor and Montgomery \cite{Bor_1992, Bor_Montgomery_1990}, Sadun and Segert \cite{Sadun_1994, Sadun_Segert_1991, Sadun_Segert_1992cmp, Sadun_Segert_1992cpam, Sadun_Segert_1993}, and other four-dimensional manifolds by work of Gritsch \cite{Gritsch_2000} and Parker \cite{Parker_1992invent}.

\subsection{Gap between energies of absolute minima and non-minimal critical points}
\label{subsec:Yang-Mills_energy_gap}
In view of the preceding discussion and the non-compactness of $\sB(P,g)$ and $M(P,g)$, it is natural to ask whether or not there is a  \emph{positive uniform gap} between the energy, $\sE_g(A)$, of points $[A]$ in the stratum, $M(P,g)$, of absolute minima of $\sE_g$ on $\sB(P,g)$ and energies of points in the strata in $\sB(P,g)$ of non-minimal critical points.

The earliest result of this kind is due to Bourguignon and Lawson \cite[Theorem D]{Bourguignon_Lawson_1981} (see also their article \cite{Bourguignon_Lawson_Simons_1979} with Simons) and that asserts that if $A$ is a Yang-Mills connection on a principal $G$-bundle over $S^4$ (with its standard round metric of radius one) such that
\[
\|F_A^+\|_{L^\infty(S^4)} < \sqrt{3},
\]
then $A$ is necessarily anti-self-dual. Their result was significantly improved by Min-Oo \cite[Theorem 2]{Min-Oo_1982} and Parker \cite[Proposition 2.2]{ParkerGauge}, by replacing the preceding $L^\infty$ condition with an $L^2$ energy condition,
\[
\|F_A^+\|_{L^2(X)} \leq \eps,
\]
where $\eps = \eps(g) \in (0,1]$ is a small enough constant and by allowing $X$ to be any closed, four-dimensional, smooth manifold equipped with a Riemannian metric, $g$, that is \emph{positive} in the sense of Definition \ref{defn:Freed_Uhlenbeck_page_174_positive_metric} below. (Min-Oo and Dodziuk \cite{Dodziuk_Min-Oo_1982}, Shen \cite{Shen_1982}, and Xin \cite{Xin_1984} also established $L^2$ energy gap results in the case of four-dimensional, non-compact, smooth manifolds equipped with complete, positive Riemannian metrics.)

\begin{rmk}[Anti-self-dual connections over the four-dimensional sphere]
\label{rmk:ASD_connections_over_4-sphere}
Anti-self-dual connections over the four-dimensional sphere, $S^4$, with its standard round metric of radius one, were completely classified and explicitly identified by Atiyah, Drinfeld, Hitchin, and Manin \cite{AtiyahGeomYM, ADHM}, for many compact Lie structure groups, $G$.
\end{rmk}

\begin{rmk}[Anti-self-dual connections over four-dimensional manifolds]
\label{rmk:ASD_connections_over_4-manifolds}
While Sedlacek \cite{Sedlacek} had employed methods of Uhlenbeck \cite{UhlLp, UhlRem}, to give a `direct method' for minimizing the Yang-Mills energy functional, the question of existence of anti-self-dual connections over four-dimensional manifolds other than $S^4$ was not settled until the advent of the work of Taubes \cite{TauSelfDual, TauIndef} and Donaldson \cite{DonASD}, drawing respectively on methods of non-linear elliptic partial differential equations and topology in the first case and on Yang-Mills gradient flow and complex analysis in the second case. Constructions of anti-self-dual connections over $\CC\PP^2$ with its Fubini-Study metric were provided by Buchdahl \cite{Buchdahl_1986} and Donaldson \cite{Donaldson_1985geomtoday}.
\end{rmk}

The key step in the proof of Min-Oo's \cite[Theorem 2]{Min-Oo_1982} is to derive an \apriori estimate (see \cite[Equation (3.15)]{Min-Oo_1982} or Corollary \ref{cor:Feehan_Leness_6-6} below) for the $W_A^{1,2}(X)$-norm of $v \in W_A^{1,2}(X; \Lambda^+\otimes\ad P)$ in terms of the $L^2(X)$-norm of $d_A^{+,*_g}v \in L^2(X; \Lambda^+\otimes\ad P)$. (See \eqref{eq:Sobolev_norms} below for definitions of Sobolev norms and spaces.) To obtain this estimate, Min-Oo employs a Bochner-Weitzenb\"ock formula for the Laplace operator, $d_A^+d_A^{+,*}$, an \apriori estimate due to P. Li \cite{Li_1980}, and the positivity condition \eqref{eq:Freed_Uhlenbeck_page_174_positive_metric} for the Riemannian metric, $g$, on $X$.

While Min-Oo's \cite[Theorem 2]{Min-Oo_1982} and Parker's \cite[Proposition 2.2]{ParkerGauge} hold for any compact Lie group, $G$, the positivity condition \eqref{eq:Freed_Uhlenbeck_page_174_positive_metric} on the Riemannian metric, $g$, imposes a strong restriction on the topology of $X$. Indeed, that positivity condition is obeyed in the case of $S^4$ with its standard round metric but there are topological obstructions to the existence of such positive Riemannian metrics on closed, four-dimensional manifolds: a necessary topological condition is that $b^+(X)=0$, where $b^\pm(X)$ denote the dimensions of the maximal positive and negative subspaces for the intersection form, $Q_X$, on $H_2(X;\RR)$ \cite[Section 1.1.1]{DK}; see Atiyah, Hitchin, and Singer \cite{AHS}, Taubes \cite{TauSelfDual, TauIndef}, or Remark \ref{rmk:Positive_metric_implies_b+_is_zero} below. However, an examination of Min-Oo's proof of his \cite[Theorem 2]{Min-Oo_1982} indicates that his key \apriori estimate \cite[Equation (3.15)]{Min-Oo_1982} is a consequence of a positive lower bound for the first eigenvalue of the Laplacian \cite[Equation (3.1)]{Min-Oo_1982},
\[
d_A^{+,g}d_A^{+,*_g} = \frac{1}{2}\left(d_Ad_A^{*_g} + d_A^{*_g}d_A\right)
\quad\hbox{ on } W_A^{2,2}(X; \Lambda^+\otimes\ad P) \subset L^2(X; \Lambda^+\otimes\ad P),
\]
that is \emph{uniform} with respect to $[A] \in \sB(P,g)$ obeying $\|F_A^{+,g}\|_{L^2(X)} < \eps$, where $\eps=\eps(g) \in (0,1]$.

The purpose of this article is to show that the positivity condition \eqref{eq:Freed_Uhlenbeck_page_174_positive_metric} on the Riemannian metric, $g$, may be replaced by a combination of mild conditions on the topology of $P$ and $X$, restriction of the Lie group, $G$, to $\SU(2)$ or $\SO(3)$, and genericity of the Riemannian metric, $g$.

\subsection{Main results}
\label{subsec:Main_results}
Because there are many potential combinations of conditions on $G$, $P$, $X$, and $g$ which imply that $\Coker d_A^{+,g} = 0$ when $A$ is anti-self-dual with respect to the Riemannian metric, $g$, it is convenient to introduce the

\begin{defn}[Good Riemannian metric]
\label{defn:Good_Riemannian_metric}
Let $G$ be a compact, simple Lie group, $X$ be a closed, connected, four-dimensional, oriented, smooth manifold, and $\eta \in H^2(X;\pi_1(G))$ be an obstruction class. We say that a Riemannian metric, $g$, on $X$ is \emph{good} if for every principal $G$-bundle, $P$, over $X$ with $\eta(P) = \eta$ and non-negative Pontrjagin degree, $\kappa(P)$, in \eqref{eq:Taubes_1982_A6} and every connection, $A$, of Sobolev class $W^{1,2}$ on $P$ with $F_A^{+,g}=0$ a.e. on $X$, then $\Coker d_A^{+,g} = 0$.
\end{defn}

See Section \ref{sec:Taubes_1982_Appendix} for the classification of principal $G$-bundles over closed, four-dimensional manifolds, $X$, when $G$ is a compact, simple Lie group. First, we observe that the proof of \cite[Theorem 2]{Min-Oo_1982}, with a minor change described in Section \ref{sec:Good_Riemannian_metrics_and_energy_gap_for_Yang-Mills_connections}, yields the following generalization.

\begin{mainthm}[$L^2$ energy gap for Yang-Mills connections over a four-dimensional manifold with a good Riemannian metric]
\label{mainthm:Min-Oo_2}
Let $G$ be a compact, simple Lie group and $P$ be a principal $G$-bundle over a closed, connected, four-dimensional, oriented, smooth manifold, $X$, with Riemannian metric, $g$, that is good in the sense of Definition \ref{defn:Good_Riemannian_metric}. Then there is a constant, $\eps = \eps(g,\kappa(P)) \in (0, 1]$, with the following significance. If $A$ is a Yang-Mills connection of Sobolev class $W^{2,2}$ on $P$, with respect to the metric $g$, such that
\begin{equation}
\label{eq:Minoo_theorem_2}
\|F_A^{+,g}\|_{L^2(X,g)} < \eps,
\end{equation}
then $A$ is anti-self-dual with respect to the metric $g$, that is, $F_A^{+,g} = 0$ a.e. on $X$.
\end{mainthm}

When $g$ is positive in the sense of Definition \ref{defn:Freed_Uhlenbeck_page_174_positive_metric}, then the forthcoming Lemma \ref{lem:Positive_metric_implies_positive_lower_bound_small_eigenvalues} implies that $g$ is good in the sense of Definition \ref{defn:Good_Riemannian_metric} and therefore Min-Oo's \cite[Theorem 2]{Min-Oo_1982} and Parker's \cite[Proposition 2.2]{ParkerGauge} are corollaries of Theorem \ref{mainthm:Min-Oo_2}. Alternatively, we may assume that $g$ is good in the sense of Definition \ref{defn:Good_Riemannian_metric} if one of the sets of conditions in Corollaries \ref{cor:Least_eigenvalue_function_positive_on_Uhlenbeck_compactification_pi1_X} or \ref{cor:Least_eigenvalue_function_positive_on_Uhlenbeck_compactification} are obeyed. Let $\pi_1(X)$ denote the fundamental group of $X$.

\begin{maincor}[$L^2$ energy gap for Yang-Mills connections over a four-dimensional manifold with a generic Riemannian metric]
\label{maincor:Min-Oo_2}
Let $G$ be a compact, simple Lie group and $P$ be a principal $G$-bundle with $\kappa(P)\geq 0$ over a closed, connected, four-dimensional, oriented, smooth manifold, $X$.  Then there is an open dense subset, $\sC(X,\kappa(P))$, of the Banach space, $\sC(X)$, of conformal equivalence classes, $[g]$, of $C^r$ Riemannian metrics on $X$ (for some integer $r\geq 3$) with the following significance. If $[g] \in \sC(X,\kappa(P))$, then there is a constant $\eps = \eps(g,\kappa(P)) \in (0, 1]$ such that the following holds. Suppose that $G$, $P$, and $X$ obey \emph{one} of the following sets of conditions:
\begin{enumerate}
\item $b^+(X) = 0$, the group $\pi_1(X)$ has no non-trivial representations in $G$, and $G = \SU(2)$ or $G = \SO(3)$; or

\item $b^+(X) > 0$, the group $\pi_1(X)$ has no non-trivial representations in $G$, and $G = \SO(3)$, and the second Stiefel-Whitney class, $w_2(P) \in H^2(X;\ZZ/2\ZZ)$, is non-trivial; or

\item $b^+(X) \geq 0$, and $G = \SO(3)$, and no principal $\SO(3)$-bundle $P_l$ over $X$ appearing in the Uhlenbeck compactification, $\bar M(P,g)$ in \eqref{eq:Uhlenbeck_compactification}, admits a flat connection.
\end{enumerate}
If $A$ is a Yang-Mills connection on $P$, with respect to the metric $g$, of Sobolev class $W^{2,2}$ that obeys \eqref{eq:Minoo_theorem_2}, then $A$ is anti-self-dual with respect to $g$, that is, $F_A^{+,g} = 0$ a.e. on $X$.
\end{maincor}

\begin{rmk}[On the hypotheses in Corollary \ref{maincor:Min-Oo_2} on $g$, $G$, $P$, and $X$]
\label{rmk:Assumptions_for_definition_Donaldson_invariants}
When $b^+(X)>1$ and $b^+(X)-b_1(X)$ is odd, where $b_1(X)$ is the first Betti number of $X$, the sets of conditions in Corollary \ref{maincor:Min-Oo_2} on the Riemannian metric, $g$, Lie group, $G$, principal $G$-bundle, $P$, and four-dimensional manifold, $X$, are those customarily employed in the most general definition of the Donaldson invariants \cite{DonPoly,DK,  KMStructure, MorganMrowkaPoly} of $X$. The possibility that Corollary \ref{maincor:Min-Oo_2} should hold is hinted at in a parenthetical remark by Taubes \cite[p. 191, second last paragraph]{TauStable}.
%COMMENT For arXiv v2, better to just explain what exactly Taubes says
\end{rmk}

\begin{rmk}[Exclusion of flat connections in the Uhlenbeck compactification of $M(P,g)$]
\label{rmk:Morgan-Mrowka_hypothesis_no_flat_connection_in_Uhlenbeck_compactification}
Despite its technical nature, the alternative `no flat connection in $\bar M(P,g)$' condition in Corollary \ref{maincor:Min-Oo_2} is easy to achieve in practice, albeit at the cost of blowing up the given four-manifold, $X$, and modifying the given principal $G$-bundle, $P$. Indeed, we recall the following facts discussed in \cite{MorganMrowkaPoly}:
\begin{enumerate}
\item If $H_1(X;\ZZ)$ has no $2$-torsion, then every principal $\SO(3)$-bundle, $\bar P$, over $X$ lifts to a principal $\U(2)$-bundle $P$, every $\SO(3)$-connection on $\bar P$ lifts to a $\U(2)$-connection on $P$, and every $\SO(3)$-gauge transformation lifts to an $\SU(2)$-gauge transformation \cite[Remark (ii), p. 225]{MorganMrowkaPoly};

\item If $\widetilde X$ is the connected sum, $X \# \CC\PP^2$, and $\widetilde P$ is the connected sum of a principal $\U(2)$-bundle, $P$, over $X$ with the principal $\U(2)$-bundle, $Q$, over $\CC\PP^2$ with $c_2(Q)=0$ and $c_1(Q) \in H^2(\CC\PP^2; \ZZ)$ equal to the Poincar\'e dual of $e = [\CC\PP^1] \in H_2(\CC\PP^2; \ZZ)$, and $\bar P$ is the principal $\SO(3)$-bundle associated to $\widetilde P$, then the following holds: No principal $\SO(3)$-bundle $P'$ over $X$ with $w_2(P') = w_2(\bar P)$ admits a flat connection
    \cite[Paragraph prior to Corollary 2.2]{MorganMrowkaPoly}.
\end{enumerate}
\end{rmk}

\begin{rmk}[Extensions to complete, four-dimensional Riemannian manifolds]
\label{rmk:Complete_Riemannian_4-manifolds}
It may be possible to extend Theorem \ref{mainthm:Min-Oo_2} and Corollary \ref{maincor:Min-Oo_2} to  to the case of a
complete, non-compact four-dimensional Riemannian manifold, thus generalizing \cite[Theorem 2]{Dodziuk_Min-Oo_1982} due to Dodziuk and Min-Oo.
\end{rmk}

Lastly, we note that we establish the following result in our companion article \cite{Feehan_yangmillsenergygapflat} by methods that are entirely different from those employed in our present article.

\begin{thm}[$L^{d/2}$-energy gap for Yang-Mills connections]
\label{thm:Main_flat}
(See \cite[Theorem 1]{Feehan_yangmillsenergygapflat}.)
Let $G$ be a compact Lie group and $P$ be a principal $G$-bundle over a closed, smooth manifold, $X$, of dimension $d \geq 2$ and endowed with a smooth Riemannian metric, $g$. Then there is a positive constant, $\eps = \eps(d,g, G) \in (0, 1]$, with the following significance.
%COMMENT For arXiv v2, the dependencies might be refined to sup bounds on Riemannian curvature, injectivity radius, Sobolev embedding constants, and the dimension of the Lie group
If $A$ is a smooth Yang-Mills connection on $P$ with respect to the metric, $g$, and its curvature, $F_A$, obeys
\begin{equation}
\label{eq:Curvature_Ldover2_small}
\|F_A\|_{L^{d/2}(X)} < \eps,
\end{equation}
then $A$ is a flat connection.
\end{thm}

Previous Yang-Mills energy gap results related to Theorem \ref{thm:Main_flat} --- due to Bourguignon, Lawson, and Simons \cite{Bourguignon_Lawson_1981, Bourguignon_Lawson_Simons_1979}, Dodziuk and Min-Oo \cite{Dodziuk_Min-Oo_1982, Min-Oo_1982}, Donaldson and Kronheimer \cite{DK}, Gerhardt \cite{Gerhardt_2010}, Shen \cite{Shen_1982}, and Xin \cite{Xin_1984} --- all require some positivity hypothesis on the curvature tensor, $\Riem_g$, of a Riemannian metric, $g$, on the manifold, $X$. The intuition underlying our proof of Theorem \ref{thm:Main_flat} is rather that an energy gap must exist because otherwise one could have non-minimal Yang-Mills connections with $L^{d/2}$-energy arbitrarily close to zero and this should violate the analyticity of the Yang-Mills $L^2$-energy functional, as manifested in the \emph{{\L}ojasiewicz-Simon gradient inequality} established by the author for arbitrary $d \geq 2$ in \cite[Theorem 23.9]{Feehan_yang_mills_gradient_flow_v2} and by R\r{a}de in \cite[Proposition 7.2]{Rade_1992} when $d=2,3$.

\subsection{Outline}
\label{subsec:Outline}
In Section \ref{sec:Taubes_1982_Appendix}, we review the classification of principal $G$-bundles, the Chern-Weil formula, and the fact that the Yang-Mills energy functional attains its absolute minimum value for a connection, $A$, on a principal $G$-bundle, $P$, with non-negative Pontrjagin degree, $\kappa(P)$, if and only if $A$ is anti-self-dual. The difficult part of the proof of Corollary \ref{maincor:Min-Oo_2} is to show that the least eigenvalue, $\mu_g(A)$, of $d_A^{+,g}d_A^{+,*_g}$ has a positive lower bound $\mu_0=\mu_0(g,\kappa(P))$ that is uniform with respect to $[A] \in \sB(P,g)$ obeying $\|F_A^{+,g}\|_{L^2(X)} < \eps$, for a small enough $\eps=\eps(g,\kappa(P)) \in (0,1]$ and under the given sets of conditions on $g$, $G$, $P$, and $X$. This step is described in Section \ref{sec:Good_Riemannian_metrics_and_eigenvalue_bounds}, where we summarize our results from our monograph \cite[Section 34]{Feehan_yang_mills_gradient_flow_v2}. We conclude in Section \ref{sec:Good_Riemannian_metrics_and_energy_gap_for_Yang-Mills_connections} with the proofs of Theorem \ref{mainthm:Min-Oo_2} and Corollary \ref{maincor:Min-Oo_2}.

\subsection{Acknowledgments}
I am grateful to Min-Oo for his encouraging comments, to the Department of Mathematics at Columbia University for their hospitality during the preparation of this article, and the anonymous referee for a careful reading of our manuscript and helpful comments.

\section{Classification of principal $G$-bundles, the Chern-Weil formula, and absolute minima of the Yang-Mills energy functional}
\label{sec:Taubes_1982_Appendix}
We summarize the main points of \cite[Section 10]{Feehan_yang_mills_gradient_flow_v2}, which extends the discussion in Donaldson and Kronheimer \cite[Sections 2.1.3 and 2.1.4]{DK} to the case of compact Lie groups, and is based in turn on Atiyah, Hitchin, and Singer \cite{AHS} and Taubes \cite{TauSelfDual}. We specialize \cite[Section 10]{Feehan_yang_mills_gradient_flow_v2} here to the case of compact, \emph{simple} Lie groups.

Given a connection, $A$, on $P$, Chern-Weil theory provides representatives for the \emph{first Pontrjagin class} of $\ad P$, namely \cite[Equation (A.7)]{TauSelfDual}
\begin{equation}
\label{eq:Taubes_1982_A7_deRham_cohomology_classes}
p_1(P) \equiv p_1(\ad P) = -\frac{1}{4\pi^2}\tr_\fg(F_A\wedge F_A) \in H^4_{\deRham}(X),
\end{equation}
and hence the \emph{first Pontrjagin number} \cite[Equation (A.7)]{TauSelfDual} (compare \cite[page 121]{Kozono_Maeda_Naito_1995}),
\begin{equation}
\label{eq:Taubes_1982_A7_scalar_and_vector_Pontrjagin_numbers}
p_1(P)[X] \equiv p_1(\ad P)[X] = -\frac{1}{4\pi^2}\int_X\tr_\fg(F_A\wedge F_A).
\end{equation}
Principal $G$-bundles, $P$, are classified \cite[Appendix]{Sedlacek}, \cite[Propositions A.1 and A.2]{TauSelfDual} by a cohomology class $\eta(P) \in H^2(X;\pi_1(G))$ and the \emph{Pontrjagin degree} of $P$ \cite[Equations (A.6) and (A.7)]{TauSelfDual},
\begin{equation}
\label{eq:Taubes_1982_A6}
\kappa(P) := -\frac{1}{r_\fg} p_1(\ad P)[X] = \frac{1}{4\pi^2 r_\fg}\int_X\tr_\fg(F_A\wedge F_A) \in \ZZ,
\end{equation}
where the positive integer, $r_\fg$, depends on the Lie group, $G$ \cite[Equation (A.5)]{TauSelfDual}; for example, if $G = \SU(n)$, then $r_\fg = 4n$.

For $G = \Or(n)$ or $\SO(n)$, then $\eta(P) = w_2(P) \in H^2(X; \ZZ/2\ZZ)$, where $w_2(P) \equiv w_2(V)$ and $V = P\times_{\Or(n)}\RR^n$ or $P\times_{\SO(n)}\RR^n$ is the real vector bundle associated to $P$ via the standard representation, $\Or(n) \hookrightarrow \GL(n;\RR)$ or $\SO(n) \hookrightarrow \GL(n;\RR)$; for $G = \U(n)$, then $\eta(P) = c_1(P) \in H^2(X; \ZZ)$, where $c_1(P) \equiv c_1(E)$ and $E = P\times_{\U(n)}\CC^n$ is the complex vector bundle associated to $P$ via the standard representation, $\U(n) \hookrightarrow \GL(n;\CC)$ \cite[Theorem 2.4]{Sedlacek}. The topological invariant, $\eta \in H^2(X; \pi_1(G))$, is the obstruction to the existence of a principal $G$-bundle, $P$ over $X$, with a specified Pontrjagin degree.

Assume in addition that $X$ is equipped with a Riemannian metric, $g$. To relate the Chern-Weil formula \eqref{eq:Taubes_1982_A7_scalar_and_vector_Pontrjagin_numbers} to the $L^2(X)$-norms of $F_A^{\pm,g}$, we need to recall some facts concerning the Killing form \cite{Knapp_2002}. Every element $\xi$ of a Lie algebra $\fg$ over a field $\KK$ defines an adjoint endomorphism, $\ad\,\xi \in \End_\KK\fg$, with the help of the Lie bracket via $(\ad\,\xi)(\zeta) := [\xi, \zeta]$, for all $\zeta \in \fg$. For a finite-dimensional Lie algebra, $\fg$, its Killing form is the symmetric bilinear form,
\begin{equation}
B(\xi, \zeta) := \tr(\ad\,\xi \circ \ad\,\zeta), \quad\forall\xi, \zeta \in \fg,
\end{equation}
with values in $\KK$. Since we restrict to compact Lie groups, their Lie algebras are real. The Lie algebra, $\fg$, is simple by hypothesis on $G$ and so its Killing form is non-degenerate. The Killing form of a semisimple Lie algebra is negative definite. For example, if $G = \SU(n)$, then $B(M, N) = 2n\tr(MN)$ for matrices $M, N \in \CC^{n\times n}$, while if $G = \SO(n)$, then $B(M, N) = (n-2)\tr(MN)$ for matrices $M, N \in \RR^{n\times n}$. In particular, if $B_\fg$ is the Killing form on $\fg$, then it defines an inner product on $\fg$ via $\langle \cdot, \cdot \rangle_\fg = -B_\fg(\cdot,\cdot)$ and thus a norm $|\cdot|_\fg$ on $\fg$.

From \eqref{eq:Donaldson_Kronheimer_2-1-25}, suppressing the metric, $g$, from our notation here for brevity, we have
$$
F_A\wedge F_A = (F_A^+ + F_A^-)\wedge (F_A^+ + F_A^-) = F_A^+ \wedge *F_A^+ - F_A^- \wedge *F_A^-.
$$
Hence,
$$
\tr_\fg(F_A\wedge F_A) = \tr_\fg(F_A^+ \wedge *F_A^+) - \tr_\fg(F_A^- \wedge *F_A^-) = \left(|F_A^-|_\fg^2 - |F_A^+|_\fg^2\right)\,d\vol,
$$
where the pointwise norm $|F_A|_\fg$ over $X$ of $F_A \in \Omega^2(X; \ad P)$ is defined by the identity,
$$
|F_A|_\fg^2\,d\vol = -\tr_\fg(F_A\wedge *F_A),
$$
and similarly for $|F_A^\pm|_\fg$. From \eqref{eq:Taubes_1982_A6}, the Pontrjagin degree of $P$ may be computed by
\begin{equation}
\label{eq:Donaldson_Kronheimer_2-1-31_Pontrjagin}
\kappa(P) = \frac{1}{4\pi^2 r_\fg} \int_X \left(|F_A^-|_\fg^2 - |F_A^+|_\fg^2\right)\,d\vol.
\end{equation}
If $A$ is self-dual, then $F_A^- \equiv 0$ over $X$ and
$$
\kappa(P) = -\frac{1}{4\pi^2 r_\fg} \int_X |F_A^+|_\fg^2\,d\vol \leq 0,
$$
while if $A$ is anti-self-dual, then $F_A^+ \equiv 0$ over $X$ and
$$
\kappa(P) = \frac{1}{4\pi^2 r_\fg} \int_X |F_A^-|_\fg^2\,d\vol \geq 0.
$$
Consequently, if $P$ admits a self-dual connection, then $\kappa(P) \leq 0$ while if $P$ admits an anti-self-dual connection, then $\kappa(P) \geq 0$. On the other hand,
\begin{align*}
\int_X |F_A|_\fg^2 \,d\vol &= \int_X \left(|F_A^+|_\fg^2 + |F_A^-|_\fg^2\right)\,d\vol
\\
&\geq \left|\int_X \left(|F_A^-|_\fg^2 - |F_A^+|_\fg^2\right)\,d\vol\right|
\\
&= 4\pi^2r_\fg |\kappa(P)| \quad\hbox{(by \eqref{eq:Donaldson_Kronheimer_2-1-31_Pontrjagin})}.
\end{align*}
Hence, $4\pi^2r_\fg |\kappa(P)|$ gives a topological lower bound for the Yang-Mills energy functional \eqref{eq:Yang-Mills_energy_functional},
$$
2\sE(A) = \int_X |F_A|_\fg^2 \,d\vol
= \int_X \left(|F_A^+|_\fg^2 + |F_A^-|_\fg^2\right) \,d\vol.
$$
If $\kappa(P) \geq 0$, then $2\sE(A)$ achieves its lower bound,
$$
\int_X \left(|F_A^-|_\fg^2 - |F_A^+|_\fg^2\right)\,d\vol = 4\pi^2r_\fg \kappa(P),
$$
if and only if
$$
\int_X \left(|F_A^+|_\fg^2 + |F_A^-|_\fg^2\right)\,d\vol
=
\int_X \left(|F_A^-|_\fg^2 - |F_A^+|_\fg^2\right)\,d\vol,
$$
that is, if and only if
$$
\int_X |F_A^+|_\fg^2 \,d\vol = 0,
$$
in other words, if and only if $F_A^+ \equiv 0$ over $X$ and $A$ is anti-self-dual. Similarly, if $\kappa(P) \leq 0$, then $2\sE(A)$ achieves its lower bound $-4\pi^2r_\fg \kappa(P)$ if and only if $F_A^- \equiv 0$ over $X$ and $A$ is self-dual.

When there is no ambiguity, we suppress explicit mention of the Lie algebra, $\fg$, in the fiber inner product and norm on $\ad P$ and write $|F_A| = |F_A|_\fg$, and so on.

\section{Good Riemannian metrics and eigenvalue bounds for Yang-Mills Laplacians}
\label{sec:Good_Riemannian_metrics_and_eigenvalue_bounds}
Consider the open neighborhood in $\sB(P,g)$ of the finite-dimensional subvariety, $M(P,g)$, defined by
\begin{equation}
\label{eq:Open_neighborhood_asd_moduli_space_FA+_L2_small}
\sB_\eps(P,g) := \{[A] \in \sB(P,g): \|F_A^{+,g}\|_{L^2(X)} < \eps\}.
\end{equation}
When the Riemannian metric, $g$, is positive in the sense of Definition \ref{defn:Freed_Uhlenbeck_page_174_positive_metric}, then the Bochner-Weitzenb\"ock formula for the Laplace operator, $d_A^{+,g}d_A^{+,*_g}$, and a simple argument (see Lemma \ref{lem:Positive_metric_implies_positive_lower_bound_small_eigenvalues} and its proof as
\cite[Lemma 34.22]{Feehan_yang_mills_gradient_flow_v2}) yields a positive lower bound for its least eigenvalue, $\mu_g(A)$, that is uniform with respect to the point $[A] \in \sB_\eps(P,g)$.

In this section, we recall from \cite[Section 34]{Feehan_yang_mills_gradient_flow_v2} how to derive a positive lower bound for $\mu_g(A)$ that is uniform with respect to the point $[A] \in \sB_\eps(P,g)$ but \emph{without} the requirement that the Riemannian metric, $g$, be positive.

\subsection{Positive Riemannian metrics and uniform positive lower bounds for the least eigenvalue of $d_A^+d_A^{+,*}$ when $F_A^+$ is $L^2$ small}
\label{subsec:Positive_Riemannian_metrics}
For a Riemannian metric $g$ on a four-dimensional, oriented manifold, $X$, let $R_g(x)$ denote its scalar curvature at a point $x \in X$ and let $\sW_g^\pm(x) \in \End(\Lambda_x^\pm)$ denote its self-dual and anti-self-dual Weyl curvature tensors at $x$, where $\Lambda_x^2 = \Lambda_x^+\oplus \Lambda_x^-$. Define
$$
w_g^\pm(x) := \text{Largest eigenvalue of } \sW_g^\pm(x), \quad\forall\, x \in X.
$$
We recall the following Bochner-Weitzenb\"ock formula \cite[Equation (6.26) and Appendix C, p. 174]{FU}, \cite[Equation (5.2)]{GroisserParkerSphere},
\begin{equation}
\label{eq:Freed_Uhlenbeck_6-26}
2d_A^{+,g}d_A^{+,*_g}v = \nabla_A^{*_g}\nabla_Av + \left(\frac{1}{3}R_g - 2w_g^+\right)v + \{F_A^{+,g}, v\},
\quad\forall\, v \in \Omega^{+,g}(X; \ad P),
\end{equation}
where $\{\cdot,\cdot\}$ denotes an algebraic bilinear operation with coefficients that depend at most on the Riemannian metric $g$ and Lie group $G$. We then make the

\begin{defn}[Positive Riemannian metric]
\label{defn:Freed_Uhlenbeck_page_174_positive_metric}
Let $X$ be a closed, four-dimensional, oriented, smooth manifold. We call a Riemannian metric, $g$, on $X$ \emph{positive} if
\begin{equation}
\label{eq:Freed_Uhlenbeck_page_174_positive_metric}
\frac{1}{3}R_g - 2w_g^+ > 0 \quad\hbox{on } X,
\end{equation}
that is, the operator $R_g/3 - 2\sW_g^+ \in \End(\Lambda^+)$ is pointwise positive definite.
\end{defn}

Of course, the simplest example of a positive metric is the standard round metric of radius one on $S^4$, where $R=1$ and $w^+=0$. Recall the

\begin{defn}[Least eigenvalue of $d_A^+d_A^{+,*}$]
\label{defn:Taubes_1982_3-1}
(See \cite[Definition 3.1]{TauSelfDual}.)
Let $G$ be a compact Lie group, $P$ be a principal $G$-bundle over a closed, four-dimensional, oriented, smooth manifold with Riemannian metric, $g$, and $A$ be a connection of Sobolev class $W^{1,2}$ on $P$. The least eigenvalue of $d_A^{+,g}d_A^{+,*_g}$ on $L^2(X; \Lambda^{+,g}\otimes\ad P)$ is
\begin{equation}
\label{eq:Least_eigenvalue_dA+dA+*}
\mu_g(A) := \inf_{v \in \Omega^{+,g}(X;\ad P)\less\{0\}}
\frac{\|d_A^{+,*_g}v\|_{L^2(X)}^2}{\|v\|_{L^2(X)}^2}.
\end{equation}
\end{defn}

If the Riemannian metric, $g$, on $X$ is positive in the sense of Definition \ref{defn:Freed_Uhlenbeck_page_174_positive_metric}, then the Bochner-Weitzenb\"ock formula \eqref{eq:Freed_Uhlenbeck_6-26} ensures that the least eigenvalue function,
\begin{equation}
\label{eq:Least_eigenvalue_dA+dA+*_function_on_MPg}
\mu_g[\,\cdot\,]: M(P,g) \to [0, \infty),
\end{equation}
defined by $\mu_g(A)$ in \eqref{eq:Least_eigenvalue_dA+dA+*}, admits a uniform positive lower bound, $\mu_0 = \mu_0(g)$,
$$
\mu_g(A) \geq \mu_0,  \quad\forall\, [A] \in M(P,g).
$$
This concept is illustrated by the following well-known elementary lemma which underlies Taubes' proof of his \cite[Theorem 1.4]{TauSelfDual}.

\begin{lem}[Positive lower bound for the least eigenvalue of $d_A^+d_A^{+,*}$ on a four-manifold with a positive Riemannian metric and $L^2$-small $F_A^+$]
\label{lem:Positive_metric_implies_positive_lower_bound_small_eigenvalues}
(See \cite[Lemma 34.22]{Feehan_yang_mills_gradient_flow_v2}.)
Let $X$ be a closed, four-dimensional, oriented, smooth manifold with Riemannian metric, $g$, that is \emph{positive} in the sense of Definition \ref{defn:Freed_Uhlenbeck_page_174_positive_metric}. Then there is a positive constant, $\eps = \eps(g) \in (0,1]$, with the following significance. Let $G$ be a compact Lie group and $P$ a principal $G$-bundle over $X$. If $A$ is a connection of Sobolev class $W^{1,2}$ on $P$ such that
$$
\|F_A^{+,g}\|_{L^2(X)} \leq \eps,
$$
and $\mu_g(A)$ is as in \eqref{eq:Least_eigenvalue_dA+dA+*}, then
\begin{equation}
\label{eq:Positive_metric_implies_positive_lower_bound_small_eigenvalues_L2small_FA+}
\mu_g(A) \geq \inf_{x\in X}\left(\frac{1}{3}R_g(x) - 2w_g^+(x)\right) > 0.
\end{equation}
\end{lem}

\begin{rmk}[Topological constraints on $X$ implied by positive Riemannian metrics]
\label{rmk:Positive_metric_implies_b+_is_zero}
The positivity hypothesis on $g$ in Lemma \ref{lem:Positive_metric_implies_positive_lower_bound_small_eigenvalues} imposes a strong constraint on the topology of $X$ since, when applied to the product connection on $X\times G$ and Levi-Civita connection on $TX$, it implies that $b^+(X) = \dim\Ker d^{+,g}d^{+,*_g} = 0$ and thus $X$ is necessarily a four-dimensional manifold with negative definite intersection form, $Q_X$, on $H^2(X;\RR)$. Indeed, we recall from \cite[Section 1.1.6]{DK} that, given \emph{any} Riemannian metric $g$ on $X$, we have an isomorphism of real vector spaces,
$$
H^2(X;\RR) \cong \sH^{+,g}(X;\RR) \oplus \sH^{-,g}(X;\RR),
$$
where  $\sH^{\pm,g}(X;\RR) := \Ker\{d^{\pm,g}d^{\pm,g,*}:\Omega^{\pm}(X;\RR)
\to \Omega^{\pm}(X;\RR)\}$, the real vector spaces of harmonic self-dual and anti-self-dual two-forms defined by the Riemannian metric, $g$, and $b^\pm(X) = \dim\sH^{\pm,g}(X)$.
\end{rmk}

Given Remark \ref{rmk:Positive_metric_implies_b+_is_zero}, we next discuss a method of ensuring a positive lower bound for the least eigenvalue of $d_A^{+,g}d_A^{+,*_g}$ that is uniform with respect to $[A] \in M(P,g)$ but which does not impose such strong restrictions on the Riemannian metric, $g$, or the topology of $X$.

\subsection{Generic Riemannian metrics and uniform positive lower bounds for the least eigenvalue of $d_A^+d_A^{+,*}$ when $F_A^+$ is identically zero}
\label{subsec:Generic_metric_theorems_Freed_Uhlenbeck}
The second approach to ensuring a uniform positive lower bound for the least eigenvalue function \eqref{eq:Least_eigenvalue_dA+dA+*_function_on_MPg} is more delicate than that of Section \ref{subsec:Positive_Riemannian_metrics} and relies on the generic metric theorems of Freed and Uhlenbeck \cite[pp. 69--73]{FU}, together with certain extensions due to Donaldson and Kronheimer \cite[Sections 4.3.3]{DK}. Under suitable hypotheses on $P$ and a generic Riemannian metric, $g$, on $X$, their results collectively ensure that $\mu_g(A) > 0$ for all $[A]$ in both $M(P,g)$ and every moduli space, $M(P_l,g)$, appearing in its \emph{Uhlenbeck compactification} (see \cite[Definition 4.4.1, Condition 4.4.2, and Theorem 4.4.3]{DK}),
\begin{equation}
\label{eq:Uhlenbeck_compactification}
\bar M(P,g) \subset \bigcup_{l=0}^L \left(M(P_l,g)\times\Sym^l(X)\right),
\end{equation}
where $L = L(\kappa(P)) \geq 0$ is a sufficiently large integer.

While the statement of \cite[Theorem 4.4.3]{DK} assumes that $G=\SU(2)$ or $\SO(3)$ --- see \cite[pages 157 and 158]{DK} --- the proof applies to any compact Lie group via the underlying analytical results due to Uhlenbeck \cite{UhlLp,UhlRem}; alternatively, one may appeal directly to the general compactness result due to Taubes \cite[Proposition 4.4]{TauPath}, \cite[Proposition 5.1]{TauFrame}. Every principal $G$-bundle, $P_l$, over $X$ appearing in \eqref{eq:Uhlenbeck_compactification} has the property that $\eta(P_l) = \eta(P)$ by \cite[Theorem 5.5]{Sedlacek}.

The generic metric theorems of Freed and Uhlenbeck \cite[pp. 69--73]{FU} and Donaldson and Kronheimer \cite[Sections 4.3.3]{DK} are normally phrased in terms of existence of a Riemannian metric, $g$, on $X$ such that $\Coker d_A^{+,g}=0$ for all $[A] \in M(P,g)$, a property of $g$ that is equivalent to $\mu_g(A) > 0$ for all $[A] \in M(P,g)$, as we shall write in the following restatement of their results.

\begin{thm}[Generic metrics theorem for simply-connected four-manifolds]
\label{thm:Donaldson-Kronheimer_Corollary_4-3-15_and_18_and_Proposition_4-3-20}
(See \cite[Theorem 34.23]{Feehan_yang_mills_gradient_flow_v2}.)
Let $G$ be a compact, simple Lie group and $P$ be a principal $G$-bundle over a closed, connected, smooth, four-dimensional manifold, $X$. Then there is an open dense subset, $\sC(X,\kappa(P))$, of the Banach space, $\sC(X)$, of conformal equivalence classes, $[g]$, of $C^r$ Riemannian metrics on $X$ (for some integer $r\geq 3$) with the following significance. Assume that $[g] \in \sC(X,\kappa(P))$ and $\pi_1(X)$ is trivial and at least \emph{one} of the following holds:
\begin{enumerate}
  \item $b^+(X) = 0$ and $G = \SU(2)$ or $G = \SO(3)$; or

  \item $b^+(X) > 0$, and $G = \SO(3)$, and the second Stiefel-Whitney class, $w_2(P) \in H^2(X;\ZZ/2\ZZ)$, is non-trivial;
\end{enumerate}
Then every point $[A] \in M(P,g)$ has the property that $\mu_g(A) > 0$.
\end{thm}

The hypothesis in Theorem \ref{thm:Donaldson-Kronheimer_Corollary_4-3-15_and_18_and_Proposition_4-3-20} that the four-manifold $X$ is simply-connected may be relaxed.

\begin{cor}[Generic metrics theorem for four-manifolds with no non-trivial representations of $\pi_1(X)$ in $G$]
\label{cor:Donaldson-Kronheimer_Corollary_4-3-15_and_18_and_Proposition_4-3-20}
(See \cite[Corollary 34.24]{Feehan_yang_mills_gradient_flow_v2}.)
Assume the hypotheses of Theorem \ref{thm:Donaldson-Kronheimer_Corollary_4-3-15_and_18_and_Proposition_4-3-20}, except replace the hypothesis that $X$ is simply-connected by the requirement that the fundamental group, $\pi_1(X)$, has no non-trivial representations in $G$. Then the conclusions of Theorem \ref{thm:Donaldson-Kronheimer_Corollary_4-3-15_and_18_and_Proposition_4-3-20} continue to hold.
\end{cor}

Our results in \cite[Section 34.3]{Feehan_yang_mills_gradient_flow_v2} ensure the continuity of $\mu_g[\,\cdot\,]$ with respect to the Uhlenbeck topology. Plainly, in Definition \ref{defn:Good_Riemannian_metric}, it suffices to consider the principal $G$-bundles, $P_l$, appearing in the space \eqref{eq:Uhlenbeck_compactification} containing the Uhlenbeck compactification, $\bar M(P,g)$. We now recall the results required from \cite[Section 34]{Feehan_yang_mills_gradient_flow_v2} that ensure that a generic Riemannian metric, $g$, is good under mild hypotheses on the topology of $P$ and $X$, provided $G=\SU(2)$ or $\SO(3)$.

\begin{thm}[Positive lower bound for the least eigenvalue of $d_A^+d_A^{+,*}$ on a four-manifold with a good Riemannian metric and anti-self-dual connection $A$]
\label{thm:Good_metric_implies_positive_lower_bound_small_eigenvalues}
(See \cite[Theorem 34.26]{Feehan_yang_mills_gradient_flow_v2}.)
Let $G$ be a compact, simple Lie group and $P$ be a principal $G$-bundle over a closed, four-dimensional, oriented, smooth manifold, $X$, with Riemannian metric, $g$. Assume that $g$ is \emph{good} in the sense of Definition \ref{defn:Good_Riemannian_metric}. Then there is a positive constant, $\mu_0 = \mu_0(g,\kappa(P))$ with the following significance. If $A$ is connection of Sobolev class $W^{1,2}$ on $P$ such that
$$
F_A^{+,g} = 0 \quad\hbox{a.e. on } X,
$$
and $\mu_g(A)$ is as in \eqref{eq:Least_eigenvalue_dA+dA+*}, then
\begin{equation}
\label{eq:Positive_metric_implies_positive_lower_bound_small_eigenvalues_good_Riem_metric_ASD_conn}
\mu_g(A) \geq \mu_0.
\end{equation}
\end{thm}

The conclusion in Theorem \ref{thm:Good_metric_implies_positive_lower_bound_small_eigenvalues} is a consequence of the facts that $\bar M(P,g)$ is compact, the extension,
\begin{equation}
\label{eq:Least_eigenvalue_function_on_Uhlenbeck_compactification}
\bar\mu_g[\,\cdot\,]: \bar M(P,g) \ni ([A],\bx) \to \mu_g[A] \in [0, \infty),
\end{equation}
to $\bar M(P,g)$ of the function \eqref{eq:Least_eigenvalue_dA+dA+*_function_on_MPg} defined by \eqref{eq:Least_eigenvalue_dA+dA+*} is continuous with respect to the Uhlenbeck topology on $\bar M(P,g)$ by Proposition \ref{prop:Lp_loc_continuity_least_eigenvalue_wrt_connection}, the fact that $\mu_g(A) > 0$ for $[A] \in M(P_l,g)$ and $P_l$ a principal $G$-bundle over $X$ appearing in the space \eqref{eq:Uhlenbeck_compactification} containing the Uhlenbeck compactification, $\bar M(P,g)$, and $g$ is good by hypothesis.

\subsection{Generic Riemannian metrics and uniform positive lower bounds for the least eigenvalue of $d_A^+d_A^{+,*}$ when $F_A^+$ is $L^2$ small}
\label{subsec:Generic_metric_theorems_Freed_Uhlenbeck_almost_ASD}
Unlike Lemma \ref{lem:Positive_metric_implies_positive_lower_bound_small_eigenvalues}, Theorem \ref{thm:Good_metric_implies_positive_lower_bound_small_eigenvalues} requires that $F_A^+=0$ a.e. on $X$, not merely that $\|F_A^{+,g}\|_{L^2(X)} < \eps$, for a small enough $\eps(g) \in (0,1]$. However, by appealing to results of Sedlacek \cite{Sedlacek}, we obtained in \cite[Section 34]{Feehan_yang_mills_gradient_flow_v2} an extension of Theorem \ref{thm:Good_metric_implies_positive_lower_bound_small_eigenvalues} that replaces the condition $F_A^{+,g}=0$ a.e. on $X$ by $\|F_A^{+,g}\|_{L^2(X)} \leq \eps$, for a small enough $\eps(g,\kappa(P)) \in (0,1]$.

\begin{thm}[Positive lower bound for the least eigenvalue of $d_A^+d_A^{+,*}$ on a four-manifold with a good Riemannian metric and almost anti-self-dual connection $A$]
\label{thm:Good_metric_implies_positive_lower_bound_small_eigenvalues_almost_ASD}
(See \cite[Theorem 34.27]{Feehan_yang_mills_gradient_flow_v2}.)
Let $G$ be a compact, simple Lie group and $P$ be a principal $G$-bundle over a closed, four-dimensional, oriented, smooth manifold, $X$, with Riemannian metric, $g$. Assume that $g$ is \emph{good} in the sense of Definition \ref{defn:Good_Riemannian_metric}. Then there is a positive constant, $\eps = \eps(g,\kappa(P)) \in (0,1]$, with the following significance. If $A$ is a connection of Sobolev class $W^{1,2}$ on $P$ such that
$$
\|F_A^{+,g}\|_{L^2(X)} \leq \eps,
$$
and $\mu_g(A)$ is as in \eqref{eq:Least_eigenvalue_dA+dA+*}, then
\begin{equation}
\label{eq:Positive_metric_implies_positive_lower_bound_small_eigenvalues_good_Riem_metric_almost_ASD_conn}
\mu_g(A) \geq \frac{\mu_0}{2},
\end{equation}
where $\mu_0 = \mu_0(g,\kappa(P))$ is the positive constant in Theorem \ref{thm:Good_metric_implies_positive_lower_bound_small_eigenvalues}.
\end{thm}

The generic metric theorems of Freed and Uhlenbeck \cite{FU}, together with their extensions due to Donaldson and Kronheimer \cite{DK}, now yield the required hypothesis in Theorems \ref{thm:Good_metric_implies_positive_lower_bound_small_eigenvalues} and \ref{thm:Good_metric_implies_positive_lower_bound_small_eigenvalues_almost_ASD} that the Riemannian metric, $g$, is good without the assumption that it is positive in the sense of Definition \ref{defn:Freed_Uhlenbeck_page_174_positive_metric} and hence yields the following corollaries.

\begin{cor}[Uniform positive lower bound for the smallest eigenvalue function when $g$ is generic, $G$ is $\SU(2)$ or $\SO(3)$, and $\pi_1(X)$ has no non-trivial representations in $G$]
\label{cor:Least_eigenvalue_function_positive_on_Uhlenbeck_compactification_pi1_X}
(See \cite[Corollary 34.28]{Feehan_yang_mills_gradient_flow_v2}.)
Assume the hypotheses of Corollary \ref{cor:Donaldson-Kronheimer_Corollary_4-3-15_and_18_and_Proposition_4-3-20} and that $g$ is generic. Then there are constants, $\eps = \eps(g,\kappa(P)) \in (0,1]$ and $\mu_0 = \mu_0(g,\kappa(P)) > 0$, such that
\begin{align*}
\mu_g(A) &\geq \mu_0, \quad\forall\, [A] \in M(P,g),
\\
\mu_g(A) &\geq \frac{\mu_0}{2}, \quad\forall\, [A] \in \sB_\eps(P,g).
\end{align*}
\end{cor}

Corollary \ref{cor:Least_eigenvalue_function_positive_on_Uhlenbeck_compactification_pi1_X} follows from the observation that $g \in \cap_{l=1}^L \sC(X,\kappa(P_l))$, where $\sC(X,\kappa(P_l))$ is as in Theorem \ref{thm:Donaldson-Kronheimer_Corollary_4-3-15_and_18_and_Proposition_4-3-20}, and thus $g$ is good in the sense of Definition \ref{defn:Good_Riemannian_metric}, together with Theorem \ref{thm:Good_metric_implies_positive_lower_bound_small_eigenvalues}. The proof of Corollary \ref{cor:Least_eigenvalue_function_positive_on_Uhlenbeck_compactification_pi1_X} extends without change to give

\begin{cor}[Uniform positive lower bound for the smallest eigenvalue function when $g$ is generic and $G$ is $\SO(3)$]
\label{cor:Least_eigenvalue_function_positive_on_Uhlenbeck_compactification}
(See \cite[Corollary 34.29]{Feehan_yang_mills_gradient_flow_v2}.)
Assume the hypotheses of Corollary \ref{cor:Least_eigenvalue_function_positive_on_Uhlenbeck_compactification_pi1_X}, but replace the hypothesis on $\pi_1(X)$ by the requirement that $G = \SO(3)$ and no principal $\SO(3)$-bundle $P_l$ over $X$ appearing in the Uhlenbeck compactification, $\bar M(P,g)$ in \eqref{eq:Uhlenbeck_compactification}, supports a flat connection. Then the conclusions of Corollary \ref{cor:Least_eigenvalue_function_positive_on_Uhlenbeck_compactification_pi1_X} continue to hold.
\end{cor}

See Remark \ref{rmk:Morgan-Mrowka_hypothesis_no_flat_connection_in_Uhlenbeck_compactification} for a discussion of the apparently technical hypothesis in Corollary \ref{cor:Least_eigenvalue_function_positive_on_Uhlenbeck_compactification} on the exclusion of flat connections in $\bar M(P,g)$.

\begin{rmk}[Kronheimer-Mrowka extension of generic metrics theorems to non-simply-connected four-dimensional manifolds]
\label{rmk:Kronheimer-Mrowka_enhancement_generic_metrics_theorem_non-simply-connected}
An enhancement due to Kronheimer and Mrowka \cite[Corollary 2.5]{KMStructure} of the generic metrics theorems of Freed and Uhlenbeck \cite[pp. 69--73]{FU} and Donaldson and Kronheimer \cite[Sections 4.3.3, 4.3.4, and 4.3.5]{DK}, for a generic Riemannian metric, $g$, on $X$ ensures the following holds even when $\pi_1(X)$ is non-trivial: For $b^+(X) \geq 0$, every connection, $A$, that is anti-self-dual with respect to $g$ on a principal $\SO(3)$ bundle over $X$ has $\Coker d_A^{+,g} = 0$, unless $A$ is reducible or flat\footnote{Kronheimer and Mrowka assume in their \cite[Definition 2.1]{KMStructure} of an \emph{admissible} four-dimensional manifold, $X$, that $b^+(X)-b^1(X)$ is odd and $b^+(X)>1$, but do not use those constraints in their proof of \cite[Corollary 2.5]{KMStructure}.}.
\end{rmk}

\section{Good Riemannian metrics and energy gap for Yang-Mills connections}
\label{sec:Good_Riemannian_metrics_and_energy_gap_for_Yang-Mills_connections}
It remains to apply the argument employed by Min-Oo in \cite[Section 3]{Min-Oo_1982} and extend his proof of \cite[Theorem 2]{Min-Oo_1982} from the case of a positive Riemannian metric to the more general case of a good Riemannian metric, thus establishing Theorem \ref{mainthm:Min-Oo_2} and hence Corollary \ref{maincor:Min-Oo_2}.

\subsection{An \apriori estimate for the operator $d_A^{+,*}$}
\label{subsec:Apriori_estimates_dA+*_and_dA+d_A+*}
For $u \in L^r(X;\Lambda^p\otimes\ad P)$, where $1 \leq r < \infty$ and $p\geq 0$ is an integer, we denote
\begin{equation}
\label{eq:Sobolev_norms}
\|u\|_{W_A^{k,r}(X)} := \left(\sum_{j=0}^k \int_X |\cov_A^ju|^r \,d\vol_g \right)^{1/r},
\end{equation}
by analogy with \cite[Definitions 2.2 and 2.3]{Aubin_1998}, where $\nabla_A$ is the covariant derivative induced by the connection, $A$, on $P$ and the Levi-Civita connection defined by the Riemannian metric, $g$, on $T^*X$, and all associated vector bundles over $X$. The Sobolev spaces, $W_A^{k,r}(X; \Lambda^p\otimes\ad P)$, are the completions of $\Omega^p(X;\ad P)$ with respect to the norms \eqref{eq:Sobolev_norms}, while the Sobolev spaces, $W_A^{k,r}(X; \Lambda^\pm\otimes\ad P)$, are the corresponding completions of $\Omega^\pm(X;\ad P)$; when $r=2$, we abbreviate these Sobolev spaces by $H_A^k(X; \Lambda^p\otimes\ad P)$ and $H_A^k(X; \Lambda^\pm\otimes\ad P)$, respectively.

We recall the following useful \apriori estimate from \cite[Lemma 6.6]{FLKM1}, based in turn on estimates due to Taubes in \cite[Lemma 5.2]{TauSelfDual} and in \cite[Appendix A]{TauIndef}.

\begin{lem}[An \apriori $L^2$ estimate for $d_A^{+,*}$ and $L^{4/3}$ estimate for $d_A^+d_A^{+,*}$]
\label{lem:Feehan_Leness_6-6}
(See \cite[Lemma 6.6]{FLKM1}.)
Let $X$ be a closed, four-dimensional, oriented, smooth manifold with Riemannian metric, $g$. Then there are positive constants, $c = c(g)$ and $\eps = \eps(g) \in (0,1]$, with the following significance. If $G$ is a compact Lie group, $A$ is a connection of Sobolev class $W^{2,2}$ on a principal $G$-bundle $P$ over $X$ with
\begin{equation}
\label{eq:L2norm_FA+_leq_small}
\|F_A^+\|_{L^2(X)} \leq \eps,
\end{equation}
and $v \in \Omega^+(X;\ad P)$, then\footnote{We correct a typographical error in the statement of inequality (2) in \cite[Lemma 6.6]{FLKM1}, where the term $\|v\|_{L^2(X)}$ was omitted on the right-hand side.}
\begin{align}
\label{eq:Feehan_Leness_6-6-1_H1Av_L2normdA+*v}
%COMMENT For arXiv v2, add \left ... \right to scale parentheses
\|v\|_{H_A^1(X)}
&\leq
c\left(\|d_A^{+,*}v\|_{L^2(X)}+\|v\|_{L^2(X)}\right),
\\
\label{eq:Feehan_Leness_6-6-1_L4v_L2dA+*v}
\|v\|_{L^4(X)}
&\leq
c\left(\|d_A^{+,*}v\|_{L^2(X)}+\|v\|_{L^2(X)}\right),
\\
\label{eq:Feehan_Leness_6-6-2_L2dA+*v_L4over3dA+dA+*v}
\|d_A^{+,*}v\|_{L^2(X)}
&\leq
c\left(\|d_A^+d_A^{+,*}v\|_{L^{4/3}(X)}+\|v\|_{L^2(X)}\right),
\\
\label{eq:Feehan_Leness_6-6-2_H1Av_L4over3dA+dA+*v}
\|v\|_{H_A^1(X)}
&\leq
c\left(\|d_A^+d_A^{+,*}v\|_{L^{4/3}(X)}+\|v\|_{L^2(X)}\right).
\end{align}
\end{lem}

\begin{proof}
The \apriori estimates \eqref{eq:Feehan_Leness_6-6-1_H1Av_L2normdA+*v} and \eqref{eq:Feehan_Leness_6-6-2_L2dA+*v_L4over3dA+dA+*v} are given by \cite[Lemma 6.6]{FLKM1} and \eqref{eq:Feehan_Leness_6-6-2_H1Av_L4over3dA+dA+*v} is a trivial consequence of those. The \apriori estimate \eqref{eq:Feehan_Leness_6-6-1_L4v_L2dA+*v} is obtained by combining \eqref{eq:Feehan_Leness_6-6-1_H1Av_L2normdA+*v} with the Kato Inequality \cite[Equation (6.20)]{FU} and the Sobolev embedding $H^1(X) \hookrightarrow L^4(X)$.
\end{proof}

See \cite[Section 34.2]{Feehan_yang_mills_gradient_flow_v2} for many more \apriori estimates of this kind. In our application, we shall only need the \apriori estimate \eqref{eq:Feehan_Leness_6-6-1_H1Av_L2normdA+*v}. We note the following immediate generalization of the \apriori estimate \cite[Equation (3.15)]{Min-Oo_1982} from the case of a positive Riemannian metric to that of a Riemannian metric, $g$, such that $\mu_g(A)>0$.

\begin{cor}[An \apriori $L^2$ estimate for $d_A^{+,*}$]
\label{cor:Feehan_Leness_6-6}
Let $X$ be a closed, four-dimensional, oriented, smooth manifold with a good Riemannian metric, $g$. Then there are positive constants, $c = c(g) \in [1,\infty)$ and $\eps = \eps(g) \in (0,1]$, with the following significance. Let $G$ be a compact Lie group and $A$ be a connection of Sobolev class $W^{1,2}$ on a principal $G$-bundle $P$ over $X$ obeying \eqref{eq:L2norm_FA+_leq_small} and $\mu_g(A)>0$. If $v \in W_A^{1,2}(X;\Lambda^{+,g}\otimes\ad P)$, then
\begin{equation}
\label{eq:Feehan_Leness_6-6-1_H1Av_L2normdA+*v_good_g}
\|v\|_{W_A^{1,2}(X)}
\leq
c\left(1 + 1/\sqrt{\mu_g(A)}\right)\|d_A^{+,*_g}v\|_{L^2(X)},
\end{equation}
where $\mu_g(A)$ is as in Definition \ref{defn:Taubes_1982_3-1}.
\end{cor}

\begin{proof}
The conclusion follows from Lemma \ref{lem:Feehan_Leness_6-6} and the Definition \ref{defn:Taubes_1982_3-1} of $\mu_g(A)$ since
\[
\sqrt{\mu_g(A)}\,\|v\|_{L^2(X)} \leq \|d_A^{+,*_g}v\|_{L^2(X)}.
\]
This completes the proof.
\end{proof}

\subsection{Completion of the proofs of Theorem \ref{mainthm:Min-Oo_2} and Corollary \ref{maincor:Min-Oo_2}}
We now have all the ingredients required to conclude the

\begin{proof}[Proof of Theorem \ref{mainthm:Min-Oo_2}]
For a connection, $A$, of class $W^{2,2}$ on $P$ with $\|F_A^{+,g}\|_{L^2(X)} < \eps$, where $\eps = \eps(g,\kappa(P)) \in (0,1]$ is as in the hypotheses of Corollary \ref{cor:Feehan_Leness_6-6}, we can apply the \apriori estimate \eqref{eq:Feehan_Leness_6-6-1_H1Av_L2normdA+*v_good_g} to $v = F_A^+$ to obtain
\begin{equation}
\label{eq:L2norm_FAplus_bounded_by_L2norm_dA*FAplus}
\|F_A^{+,g}\|_{L^2(X)} \leq c\left(1 + 1/\sqrt{\mu_0/2}\right)\|d_A^{+,*_g}F_A^{+,g}\|_{L^2(X)},
\end{equation}
where $c=c(g)$ is as in Corollary \ref{cor:Feehan_Leness_6-6} and $\mu_0 = \mu_0(g,\kappa(P))$ is the uniform positive lower bound for $2\mu_g(A)$ provided by Theorem \ref{thm:Good_metric_implies_positive_lower_bound_small_eigenvalues_almost_ASD}. Just as in the paragraph following \cite[Equation (3.15)]{Min-Oo_1982}, we have
\[
d_A^{+,*_g}F_A^{+,g} = -\frac{1}{2}*_{g}d_{A}*_{g}(1+*_g)F_{A} = \frac{1}{2}d_A^{*_g}F_A.
\]
By hypothesis, $A$ is Yang-Mills with respect to $g$, thus $d_A^{*_g}F_A = 0$ a.e. on $X$, and so
\[
d_A^{+,*_g}F_A^{+,g} = 0 \quad\hbox{a.e. on } X.
\]
Therefore, $F_A^{+,g} = 0$ a.e. on $X$ by \eqref{eq:L2norm_FAplus_bounded_by_L2norm_dA*FAplus} and $A$ is anti-self-dual with respect to $g$.
\end{proof}

\begin{proof}[Proof of Corollary \ref{maincor:Min-Oo_2}]
The conclusions follow from Theorem \ref{mainthm:Min-Oo_2} and the positive uniform lower bound on $\mu_g(A)$ provided by Corollaries \ref{cor:Least_eigenvalue_function_positive_on_Uhlenbeck_compactification_pi1_X} or \ref{cor:Least_eigenvalue_function_positive_on_Uhlenbeck_compactification}.
\end{proof}

\appendix

\section{Uhlenbeck continuity of the least eigenvalue of $d_A^+d_A^{+,*}$ with respect to the connection}
\label{sec:Continuity_least_eigenvalue_d_A+d_A+*_wrt_connection}
For completeness, we include the statement and proof of our \cite[Proposition 34.14]{Feehan_yang_mills_gradient_flow_v2} (restated here as Proposition \ref{prop:Lp_loc_continuity_least_eigenvalue_wrt_connection}) from \cite[Section 34.3]{Feehan_yang_mills_gradient_flow_v2} together with required preparatory lemmata.

In order to extend Lemma \ref{lem:Positive_metric_implies_positive_lower_bound_small_eigenvalues} from the simple case of a \emph{positive} Riemannian metric, $g$, though arbitrary compact Lie group, $G$, to the more difficult case of a \emph{generic} Riemannian metric, $g$, and Lie groups $G=\SU(2)$ or $\SO(3)$, we shall need to closely examine the continuity properties of the least eigenvalue of the elliptic operator $d_A^+d_A^{+,*}$ with respect to the connection, $A$. We begin by recalling the following $L^p$ analogue of the \apriori $L^\infty$ estimate \cite[Lemma 5.3, Item (1)]{FeehanSlice}.

\begin{lem}[An \apriori $L^p$ estimate for the connection Laplace operator]
\label{lem:Feehan_5-3-1_Lp}
(See \cite[Lemma 34.5]{Feehan_yang_mills_gradient_flow_v2}.)
Let $X$ be a closed, smooth manifold of dimension $d\geq 4$ and Riemannian metric, $g$, and $q \in (d,\infty)$. Then there is a positive constant, $c = c(g,q)$, with the following significance. Let $r \in (d/3, d/2)$ be defined by $1/r = 2/d + 1/q$. Let $A$ be a Riemannian connection of class $C^\infty$ on a Riemannian vector bundle $E$ over $X$ with covariant derivative $\nabla_A$ and curvature $F_A$. If $v \in C^\infty(X;E)$, then
\begin{equation}
\label{eq:Feehan_5-3-1_dimension_Lp}
\|v\|_{L^q(X)}
\leq
c\left(\|\cov_A^*\cov_Av\|_{L^r(X)} + \|v\|_{L^r(X)}\right).
\end{equation}
\end{lem}

\begin{proof}
We adapt the proof of the estimate \cite[Lemma 5.3, Item (1)]{FeehanSlice}. For any $v\in C^\infty(X;E)$, we have the pointwise identity \cite[Equation (6.18)]{FU}, namely
$$
|\cov_Av|^2 + \frac{1}{2} d^*d|v|^2 = \langle\cov_A^*\cov_A v,v\rangle \quad\hbox{on } X,
$$
and thus,
$$
|\cov_Av|^2 + \frac{1}{2}(1+ d^*d)|v|^2 = \langle\cov_A^*\cov_A v,v\rangle + \frac{1}{2}|v|^2
\quad\hbox{on }X.
$$
As in \cite[Section 5.1]{FeehanSlice}, we let $\cG \in C^\infty(X\times X\less\Delta;\mathbb{R})$ denote the Green kernel for the augmented Laplace operator, $d^*d+1$, on $C^\infty(X;\mathbb{R})$, where $\Delta$ denotes the diagonal of $X\times X$. Using the preceding identity and the fact that
$$
\int_X \cG(x,\cdot)(d^*d+1)|v|^2\,dV = |v|^2(x), \quad\forall\, x \in X,
$$
we obtain
\begin{multline*}
\int_X \cG(x,\cdot)|\cov_Av|^2\,d\vol + \frac{1}{2} |v|^2(x)
\\
\leq
\int_X G(x,\cdot)|\langle \cov_A^*\cov_Av,v\rangle|\,d\vol
+ \frac{1}{2}\int_X G(x,\cdot)|v|^2\,d\vol, \quad\forall\, x \in X.
\end{multline*}
Writing the Green operator, $\cG := (d^*d + 1)^{-1}$, as
$$
(\cG v)(x) := \int_X \cG(x,\cdot)v\,d\vol, \quad\forall\, x \in X,
$$
we observe that $\cG$ extends to define a bounded operator,
$$
\cG:L^s(X) \to L^t(X),
$$
when $s \in (1, d/2)$ and $t\in (d/2,\infty)$ satisfy $1/s = 2/d + 1/t$ since \afortiori $\cG$ extends to a bounded operator,
$$
\cG:L^s(X) \to W^{2,s}(X),
$$
and $W^{2,s}(X) \hookrightarrow L^t(X)$ is a continuous embedding by \cite[Theorem 4.12]{AdamsFournier} when $t$ is defined as above. In particular, there is a positive constant, $c=c(g,s)$, such that
$$
\|\cG f\|_{L^t(X)} \leq c\|f\|_{L^s(X)}, \quad \forall\, f \in L^s(X).
$$
Therefore, expressing the preceding inequality for $v$ more compactly as
$$
\cG(|\cov_Av|^2) + \frac{1}{2} |v|^2
\leq
\cG|\langle \cov_A^*\cov_Av,v\rangle| + \frac{1}{2}\cG(|v|^2) \quad\hbox{on } X,
$$
and dropping the first term on the left, we find that, for $c = c(g,s)$,
\begin{align*}
\||v|^2\|_{L^t(X)}
&\leq
2\|\cG|\langle \cov_A^*\cov_Av,v\rangle|\|_{L^t(X)}
+ \|\cG(|v|^2)\|_{L^t(X)}
\\
&\leq c\|\langle \cov_A^*\cov_Av,v\rangle\|_{L^s(X)} + c\||v|^2\|_{L^s(X)}.
\end{align*}
Using
$$
\frac{1}{s} = \frac{2}{d} + \frac{1}{t} = \frac{2}{d} + \frac{1}{2t} + \frac{1}{2t}
= \frac{d+4t}{2dt} + \frac{1}{2t},
$$
we see that
\begin{align*}
\|\langle \cov_A^*\cov_Av,v\rangle\|_{L^s(X)}
&\leq
\|\cov_A^*\cov_Av\|_{L^{2dt/(d+4t)}(X)} \|v\|_{L^{2t}(X)},
\\
\||v|^2\|_{L^s(X)} &\leq \|v\|_{L^{2dt/(d+4t)}(X)} \|v\|_{L^{2t}(X)}.
\end{align*}
Therefore,
$$
\|v\|_{L^{2t}(X)}^2 \leq c\|\cov_A^*\cov_Av\|_{L^{2dt/(d+4t)}(X)} \|v\|_{L^{2t}(X)}
+ c\|v\|_{L^{2dt/(d+4t)}(X)} \|v\|_{L^{2t}(X)},
$$
and thus, for $v$ not identically zero,
$$
\|v\|_{L^{2t}(X)}
\leq
c\left(\|\cov_A^*\cov_Av\|_{L^{2dt/(d+4t)}(X)} + \|v\|_{L^{2dt/(d+4t)}(X)}\right),
$$
for any $t \in (d/2,\infty)$. But $(d+4t)/(2dt)=  1/(2t) + 2/d$, so writing $q = 2t \in (d,\infty)$ and $r = 2dt/(d+4t) \in (d/3, d/2)$ (for $t \in (d/2,\infty)$) yields \eqref{eq:Feehan_5-3-1_dimension_Lp}, where $1/r = 2/d + 1/q$.
\end{proof}

We now apply Lemma \ref{lem:Feehan_5-3-1_Lp} to prove the

\begin{lem}[An \apriori $L^p$ estimate for $d_A^+d_A^{+,*}$]
\label{lem:Feehan_5-3-1_Lp_dA+dA+*}
(See \cite[Lemma 34.6]{Feehan_yang_mills_gradient_flow_v2}.)
Let $X$ be a closed, four-dimensional, oriented, smooth manifold, $X$, with Riemannian metric, $g$, and $q \in [4,\infty)$. Then there are positive constants, $c = c(g,q)\in [1,\infty)$ and $\eps = \eps(g,q) \in (0,1]$, with the following significance. Let $r \in [4/3, 2)$ be defined by $1/r = 1/2 + 1/q$. Let $G$ be a compact Lie group and $A$ a connection of class $C^\infty$ on a principal bundle $P$ over $X$ that obeys the curvature bound \eqref{eq:L2norm_FA+_leq_small}. If $v \in \Omega^+(X;\ad P)$, then
\begin{equation}
\label{eq:Feehan_5-3-1_Lp_dA+dA+*}
\|v\|_{L^q(X)}
\leq
c\left(\|d_A^+d_A^{+,*}v\|_{L^r(X)} + \|v\|_{L^r(X)}\right).
\end{equation}
\end{lem}

\begin{proof}
We first dispose of the simplest case, $q=4$ and $r=4/3$. We combine the Kato Inequality \cite[Equation (6.20)]{FU}, Sobolev embedding $H^1(X) \hookrightarrow L^4(X)$, and \apriori estimate \eqref{eq:Feehan_Leness_6-6-2_H1Av_L4over3dA+dA+*v} to give
$$
\|v\|_{L^4(X)}
\leq
c(\|d_A^+d_A^{+,*}v\|_{L^{4/3}(X)}+\|v\|_{L^2(X)}).
$$
Substituting the interpolation inequality, $\|v\|_{L^2(X)} \leq \|v\|_{L^{4/3}(X)}^{1/2}\|v\|_{L^4(X)}^{1/2}$, in the preceding estimate yields, for any $\zeta > 0$,
\begin{align*}
\|v\|_{L^4(X)}
&\leq
c\|d_A^+d_A^{+,*}v\|_{L^{4/3}(X)} + c\|v\|_{L^{4/3}(X)}^{1/2}\|v\|_{L^4(X)}^{1/2}
\\
&\leq c\|d_A^+d_A^{+,*}v\|_{L^{4/3}(X)} + \frac{c}{2\zeta}\|v\|_{L^{4/3}(X)} + \frac{c\zeta}{2}\|v\|_{L^4(X)}.
\end{align*}
We obtain \eqref{eq:Feehan_5-3-1_Lp_dA+dA+*} when $q=4$ and $r=4/3$ from the preceding inequality by choosing $\zeta = 1/c$.

For the remainder of the proof, we assume $q \in (4,\infty)$ and $r \in (4/3,2)$. The Bochner-Weitzenb\"ock formula \eqref{eq:Freed_Uhlenbeck_6-26}, namely,
$$
2d_A^+d_A^{+,*} = \cov_A^*\cov_A + \left(\frac{R}{3} - 2w^+\right) + \{F_A^+, \cdot\},
$$
yields, for $v \in \Omega^+(X;\ad P)$,
$$
\|\cov_A^*\cov_Av\|_{L^r(X)}
\leq
2\|d_A^+d_A^{+,*}v\|_{L^r(X)}
+ c\|v\|_{L^r(X)} + \|\{F_A^+, v\}\|_{L^r(X)},
$$
and some $c = c(g)$. Since $1/r = 1/2 + 1/q$ by hypothesis, we see that
$$
\|\{F_A^+, v\}\|_{L^r(X)} \leq c\|F_A^+\|_{L^2(X)} \|v\|_{L^q(X)},
$$
for some $c = c(g)$. Combining the preceding inequalities with the estimate \eqref{eq:Feehan_5-3-1_dimension_Lp} yields
$$
\|v\|_{L^q(X)}
\leq
c\|d_A^+d_A^{+,*}v\|_{L^r(X)} + c\|v\|_{L^r(X)} + c\|F_A^+\|_{L^2(X)} \|v\|_{L^q(X)},
$$
for some $c = c(g,q)$. Provided $c\|F_A^+\|_{L^2(X)} \leq 1/2$, rearrangement gives \eqref{eq:Feehan_5-3-1_Lp_dA+dA+*} when $r \in (4/3,2)$.
\end{proof}

The following proposition extends \cite[Lemmata 34.11, 34.12, and 34.13]{Feehan_yang_mills_gradient_flow_v2} in order to accommodate the weak notion of convergence described by Sedlacek in his \cite[Theorem 3.1]{Sedlacek}. In contrast to the Uhlenbeck convergence as defined in \cite[Condition 4.4.2]{DK}, Sedlacek replaces the usual strong $W_{\loc}^{k,p}(X\less\Sigma)$ convergence of connections, with $k\geq 1$ and $p\geq 2$ obeying $kp>4$, by weak $H_{\loc}^1(X\less\Sigma)$ convergence and strong $L_{\loc}^p(X\less\Sigma)$ convergence with $p\in [2,4)$. Here, $\Sigma = \{x_1,\ldots,x_l\} \subset X$ is a finite set of points where the curvature densities, $|F_{A_m}|^2$, concentrate as $m\to\infty$. If $Y$ is a closed, smooth manifold with Riemannian metric $h$, we let $\Inj(Y,h)$ denote its injectivity radius.

\begin{prop}[$L_{\loc}^p$ continuity of the least eigenvalue of $d_A^+d_A^{+,*}$ with respect to the connection for $2\leq p < 4$]
\label{prop:Lp_loc_continuity_least_eigenvalue_wrt_connection}
(See \cite[Proposition 34.14]{Feehan_yang_mills_gradient_flow_v2}.)
Let $X$ be a closed, connected, four-dimensional, oriented, smooth manifold with Riemannian metric, $g$. Then there are a positive constant $c = c(g) \in [1, \infty)$ and a constant $\eps = \eps(g) \in (0,1]$ such that the following holds. Let $G$ be a compact Lie group, $A_0$ a connection of class $H^1$ on a principal $G$-bundle $P_0$ over $X$ obeying the curvature bound \eqref{eq:L2norm_FA+_leq_small} with constant $\eps$, and $L\geq 1$ an integer, and $p \in [2,4)$. Then there are constants $c_p = c_p(g,p) \in [1, \infty)$ and $\delta = \delta(\mu(A_0),g,L,p) \in (0,1]$ and $\rho_0 = \rho_0(\mu(A_0),g,L)\in (0, 1\wedge \Inj(X,g)]$ with the following significance. Let $\rho \in (0,\rho_0]$ and $\Sigma = \{x_1,\ldots,x_L\} \subset X$ be such that
$$
\dist_g(x_l,x_k) \geq \rho \quad\hbox{for all } k\neq l,
$$
and let $U \subset X$ be the open subset given by
$$
U := X \less \bigcup_{l=1}^L \bar B_{\rho/2}(x_l).
$$
Let $P$ be a principal $G$-bundle over $X$ such that there is an isomorphism of principal $G$-bundles, $u:P\restriction X\less\Sigma \cong P_0\restriction X\less\Sigma$, and identify $P\restriction X\less\Sigma$ with $P_0\restriction X\less\Sigma$ using this isomorphism. Let $A$ be a connection of class $H^1$ on $P$ obeying the curvature bound \eqref{eq:L2norm_FA+_leq_small} with constant $\eps$ such that
\begin{equation}
\label{eq:Lp_norm_AminusA0_U_leq_small}
\|A-A_0\|_{L^p(U)} \leq \delta.
\end{equation}
Then $\mu(A)$ in \eqref{eq:Least_eigenvalue_dA+dA+*} satisfies the \emph{lower} bound,
\begin{multline}
\label{eq:Lower_bound_sqrt_muA_for_A_Lp_loc_near_A0}
\sqrt{\mu(A)}
\geq
\sqrt{\mu(A_0)} - c\sqrt{L}\,\rho^{1/6}(\mu(A)+1)
\\
- cL\rho\left(\sqrt{\mu(A)}+1\right) - c_p\|A-A_0\|_{L^p(U)}(\mu(A) + 1),
\end{multline}
and \emph{upper} bound,
\begin{multline}
\label{eq:Upper_bound_sqrt_muA_for_A_Lp_loc_near_A0}
\sqrt{\mu(A)}
\leq
\sqrt{\mu(A_0)} + c\sqrt{L}\,\rho^{1/6}(\mu(A_0)+1)
\\
+ cL\rho\left(\sqrt{\mu(A_0)}+1\right) + c_p\|A-A_0\|_{L^p(U)}(\mu(A_0) + 1).
\end{multline}
\end{prop}

\begin{proof}
Since the argument is lengthy, we divide it into several steps.

\setcounter{step}{0}
\begin{step}[The eigenvalue identity]
By hypothesis, $X = U \cup (\cup_{l=1}^L \bar B_{\rho/2}(x_l))$ and we may choose a $C^\infty$ partition of unity, $\{\chi_l\}_{l=0}^L$, for $X$ subordinate to the open cover of $X$ given by $U$ and the open balls, $B_\rho(x_l)$ for $1\leq l\leq L$, such that $\chi_l = 1$ on $B_{\rho/2}(x_l)$ and $\supp\chi_l \subset B_\rho(x_l)$ for $1\leq l\leq L$ and $\supp\chi_0 \subset U$, while $\sum_{l=0}^L\chi_l = 1$ on $X$ and $0\leq \chi_l\leq 1$ for $0\leq l\leq L$. In addition, we may suppose that there is a constant, $c=c(g)$, such that
\begin{equation}
\label{eq:Pointwise_bound_dchi_l}
|d\chi_l| \leq \frac{c}{\rho} \quad\hbox{on } X, \quad\hbox{for } 0 \leq l \leq L.
\end{equation}
We consider $v \in \Omega^+(X;\ad P)$ and write $v = \sum_{l=0}^L\chi_l v$. Because $\supp\chi_k \cap \supp\chi_l = \emptyset$ for all $k\neq l$ with $1 \leq k,l \leq L$ and $\chi_0 = 1-\chi_l$ on $B_\rho(x_l)$ for $1\leq l \leq l$, we see that
\begin{align*}
\|d_A^{+,*}v\|_{L^2(X)}^2
&=
\sum_{l=0}^L\|d_A^{+,*}(\chi_lv)\|_{L^2(X)}^2
+ 2\sum_{k<l} \left(d_A^{+,*}(\chi_kv), d_A^{+,*}(\chi_lv)\right)_{L^2(X)}
\\
&=
\sum_{l=0}^L\|d_A^{+,*}(\chi_lv)\|_{L^2(X)}^2
+ 2\sum_{l=1}^L \left(d_A^{+,*}(\chi_0v), d_A^{+,*}(\chi_lv)\right)_{L^2(X)}
\\
&=
\sum_{l=0}^L\|d_A^{+,*}(\chi_lv)\|_{L^2(X)}^2
+ 2\sum_{l=1}^L \left(d_A^{+,*}((1-\chi_l)v), d_A^{+,*}(\chi_lv)\right)_{L^2(X)},
\end{align*}
and hence,
$$
\|d_A^{+,*}v\|_{L^2(X)}^2
= \|d_A^{+,*}(\chi_0v)\|_{L^2(X)}^2 - \sum_{l=1}^L\|d_A^{+,*}(\chi_lv)\|_{L^2(X)}^2
+ 2\sum_{l=1}^L \left(d_A^{+,*}v, d_A^{+,*}(\chi_lv)\right)_{L^2(X)}.
$$
We now choose $v \in H_A^1(X;\Lambda^+\otimes\ad P)$ with $\|v\|_{L^2(X)}=1$ to be an eigenvector for the least eigenvalue $\mu(A)$ of $d_A^+d_A^{+,*}$. Hence, the preceding identity and \eqref{eq:Least_eigenvalue_dA+dA+*} yield
\begin{equation}
\label{eq:Eigenvalue_identity_muA_for_A_Lp_loc_near_A0}
\mu(A) = \|d_A^{+,*}(\chi_0v)\|_{L^2(X)}^2 - \sum_{l=1}^L\|d_A^{+,*}(\chi_lv)\|_{L^2(X)}^2
+ 2\sum_{l=1}^L \left(d_A^{+,*}v, d_A^{+,*}(\chi_lv)\right)_{L^2(X)},
\end{equation}
the basic eigenvalue identity for $\mu(A)$.
\end{step}

To proceed further, we need a lower bound for the expression $\|d_A^{+,*}(\chi_0v)\|_{L^2(X)}^2$ in \eqref{eq:Eigenvalue_identity_muA_for_A_Lp_loc_near_A0} in terms of $\mu(A_0)$ and small upper bounds for the remaining terms on the right-hand side of \eqref{eq:Eigenvalue_identity_muA_for_A_Lp_loc_near_A0}.

\begin{step}[Upper bound for the terms $(d_A^{+,*}v, d_A^{+,*}(\chi_lv))_{L^2(X)}$ when $1\leq l\leq L$]
Observe that
$$
\left(d_A^{+,*}v, d_A^{+,*}(\chi_lv)\right)_{L^2(X)}
=
\left(d_A^+d_A^{+,*}v, \chi_lv\right)_{L^2(X)}
=
\mu(A)\left(v, \chi_lv\right)_{L^2(X)},
$$
and thus, for $1\leq l\leq L$,
\begin{align*}
\left|\left(d_A^{+,*}v, d_A^{+,*}(\chi_lv)\right)_{L^2(X)}\right|
&\leq
\mu(A) \|\sqrt{\chi_l} v\|_{L^2(X)}^2
\\
&\leq
\mu(A) \left(\Vol_g(\supp\chi_l)\right)^{1/2}\|v\|_{L^4(X)}^2
\\
&\leq c\rho^2\mu(A)\|v\|_{L^4(X)}^2
\\
&\leq c\rho^2\mu(A)\left(\|d_A^{+,*}v\|_{L^2(X)}+\|v\|_{L^2(X)}\right)^2
\quad\hbox{(by \eqref{eq:Feehan_Leness_6-6-1_L4v_L2dA+*v})}.
\end{align*}
Thus, noting that $\|d_A^{+,*}v\|_{L^2(X)} = \sqrt{\mu(A)}\|v\|_{L^2(X)}$ by \eqref{eq:Least_eigenvalue_dA+dA+*} and $\|v\|_{L^2(X)}=1$, we have
\begin{equation}
\label{eq:Upper_bound_dA*v_dA*chiv_L2innerproduct_for_A_Lp_loc_near_A0}
\left|\left(d_A^{+,*}v, d_A^{+,*}(\chi_lv)\right)_{L^2(X)}\right|
\leq
c\rho^2\mu(A)(1+\mu(A)), \quad\hbox{for } 1\leq l\leq L,
\end{equation}
where $c=c(g)$.
\end{step}

\begin{step}[Upper bound for the terms $\|d_A^{+,*}(\chi_lv)\|_{L^2(X)}^2$ when $1\leq l\leq L$]
For $0\leq l\leq L$, we have
$$
d_{A}^{+,*}(\chi_{l}v)
=
-*d_{A}*(\chi_{l}v)
=
-*(d\chi_{l}\wedge v + \chi_{l}d_{A}v)
=
-*(d\chi_{l}\wedge v) + \chi_{l}d_{A}^{+,*}v.
$$
Hence, for $s \in (2,4)$ and $t\in (4,\infty)$ obeying $1/2=1/s+1/t$, we have
$$
\left| \|d_A^{+,*}(\chi_lv)\|_{L^2(X)} - \|\chi_ld_A^{+,*}v\|_{L^2(X)} \right|
\leq
\|d\chi_l\|_{L^s(X)} \|v\|_{L^t(X)}.
$$
The pointwise bound \eqref{eq:Pointwise_bound_dchi_l} for $d\chi_l$ implies that there is a positive constant, $c = c(g)$ when $1\leq l\leq L$ and $c = Lc_0(g)$ when $l=0$, such that, for any $u \in [1,\infty]$,
\begin{equation}
\label{eq:Lp_norm_dchi_l}
\|d\chi_l\|_{L^u(X)} \leq c\rho^{(4/u)-1}, \quad 0\leq l\leq L.
\end{equation}
We choose $s=u=3$ and $t=6$ and combine the two preceding inequalities to give
\begin{align*}
{}&\left| \|d_A^{+,*}(\chi_lv)\|_{L^2(X)} - \|\chi_ld_A^{+,*}v\|_{L^2(X)} \right|
\\
&\quad \leq
c\rho^{1/3} \|v\|_{L^6(X)}
\\
&\quad \leq c\rho^{1/3} \left(\|d_A^+d_A^{+,*}v\|_{L^{3/2}(X)} + \|v\|_{L^{3/2}(X)}\right)
\quad\hbox{(by \eqref{eq:Feehan_5-3-1_Lp_dA+dA+*} with $r=3/2$)}
\\
&\quad = c\rho^{1/3}(\mu(A) + 1)\|v\|_{L^{3/2}(X)},
\end{align*}
where we used the fact that $d_A^+d_A^{+,*}v = \mu(A)v$ in the last equality. Thus, for a positive constant, $c = c(g)$ when $1\leq l\leq L$ and $c = Lc_0(g)$ when $l=0$,
\begin{equation}
\label{eq:L2norm_dA+*chi_ellv_minus_L2norm_chi_elldA+*v}
\left| \|d_A^{+,*}(\chi_lv)\|_{L^2(X)} - \|\chi_ld_A^{+,*}v\|_{L^2(X)} \right|
\leq c\rho^{1/3}(\mu(A) + 1), \quad\hbox{for } 0\leq l\leq L.
\end{equation}
Restricting now to $1\leq l\leq L$, we have
\begin{align*}
\|\chi_ld_A^{+,*}v\|_{L^2(X)}^2
&=
\left(\chi_l^2 d_A^{+,*}v, d_A^{+,*}v\right)_{L^2(X)}
\\
&=
\left(\chi_l^2 d_A^+d_A^{+,*}v
+ 2\chi_l\left(d\chi_l\wedge d_A^{+,*}v\right)^+, v\right)_{L^2(X)}
\\
&=
\mu(A)\|\chi_lv\|_{L^2(X)}^2 + 2\left(\chi_ld\chi_l\wedge d_A^{+,*}v, v\right)_{L^2(X)}
\\
&\leq \mu(A)(\Vol_g(\supp\chi_l))^{1/2}\|v\|_{L^4(X)}^2
+ 2\left(\chi_ld\chi_l\wedge d_A^{+,*}v, v\right)_{L^2(X)}.
\end{align*}
Therefore, applying \eqref{eq:Feehan_Leness_6-6-1_L4v_L2dA+*v} in the preceding inequality together with the facts that $\|d_A^{+,*}v\|_{L^2(X)} = \sqrt{\mu(A)}\|v\|_{L^2(X)}$ by \eqref{eq:Least_eigenvalue_dA+dA+*} and $\|v\|_{L^2(X)} = 1$,
\begin{equation}
\label{eq:Upper_bound_chidA*v_L2norm_for_A_Lp_loc_near_A0}
\|\chi_ld_A^{+,*}v\|_{L^2(X)}^2
\leq
c\rho^2\mu(A)\left(\sqrt{\mu(A)} + 1\right)^2
+ 2\left(\chi_ld\chi_l\wedge d_A^{+,*}v, v\right)_{L^2(X)},
\quad\hbox{for } 1\leq l\leq L.
\end{equation}
The inner product term in \eqref{eq:Upper_bound_chidA*v_L2norm_for_A_Lp_loc_near_A0} is bounded via
\begin{align*}
\left|\left(\chi_ld\chi_l\wedge d_A^{+,*}v, v\right)_{L^2(X)}\right|
&\leq
\|d\chi_l\|_{L^3(X)} \|d_A^{+,*}v\|_{L^2(X)} \|v\|_{L^6(X)}
\\
&=
\sqrt{\mu(A)}\, \|d\chi_l\|_{L^3(X)} \|v\|_{L^6(X)}
\quad\hbox{(by \eqref{eq:Least_eigenvalue_dA+dA+*})}
\\
&\leq c\rho^{1/3}\sqrt{\mu(A)}(\mu(A)+1)
\quad\hbox{(by \eqref{eq:Feehan_5-3-1_Lp_dA+dA+*} and \eqref{eq:Lp_norm_dchi_l}),}
\end{align*}
with $q=6$ and $r=3/2$ in \eqref{eq:Feehan_5-3-1_Lp_dA+dA+*}.
Hence, substituting the preceding inequality in \eqref{eq:Upper_bound_chidA*v_L2norm_for_A_Lp_loc_near_A0} yields
$$
\|\chi_ld_A^{+,*}v\|_{L^2(X)}^2
\leq
c\left(\rho^2\mu(A) + \rho^{1/3}\sqrt{\mu(A)}\right)(\mu(A)+1),
\quad\hbox{for } 1\leq l\leq L.
$$
By combining the preceding estimate with \eqref{eq:L2norm_dA+*chi_ellv_minus_L2norm_chi_elldA+*v} (and the elementary inequality, $x^2 \leq 2(x-y)^2 + 2y^2$ for $x,y\in\RR$) we obtain
\begin{align*}
\|d_A^{+,*}(\chi_lv)\|_{L^2(X)}^2
&\leq
2\left| \|d_A^{+,*}(\chi_lv)\|_{L^2(X)} - \|\chi_ld_A^{+,*}v\|_{L^2(X)} \right|^2
+ 2\|\chi_ld_A^{+,*}v\|_{L^2(X)}^2
\\
&\leq c\rho^{2/3}(\mu(A) + 1)^2
+ c\left(\rho^2\mu(A) + \rho^{1/3}\sqrt{\mu(A)}\right)(\mu(A)+1),
\end{align*}
and thus, noting that $\rho \in (0,1]$,
\begin{equation}
\label{eq:Upper_bound_dA*chiv_L2norm_for_A_Lp_loc_near_A0}
\|d_A^{+,*}(\chi_lv)\|_{L^2(X)}^2
\leq
c\rho^{1/3}(\mu(A)+1)^2, \quad\hbox{for } 1\leq l\leq L,
\end{equation}
for $c=c(g)$. This completes our analysis of all terms on the right-hand side of \eqref{eq:Eigenvalue_identity_muA_for_A_Lp_loc_near_A0} with $l\neq 0$.
\end{step}

\begin{step}[Lower bound for the term $\|d_A^{+,*}(\chi_0v)\|_{L^2(X)}$ and preliminary lower bound for $\mu(A)$]
Without loss of generality in the remainder of the proof, we may restrict attention to  $p \in (2,4]$. For convenience, we write  $a := A-A_0 \in H_{A_0}^1(X;\Lambda^1\otimes\ad P)$. For the term in \eqref{eq:Eigenvalue_identity_muA_for_A_Lp_loc_near_A0} with $l=0$, we note that $d_A^{+,*}v = d_{A_0}^{+,*}v - *(a\wedge v)$ on $X\less\Sigma$ and thus, for $p\in (2,4]$ and $q\in [4,\infty)$ defined by $1/2=1/p+1/q$ and $r\in [4/3,2)$ defined by $1/r = 1/2+1/q$,
\begin{align*}
{}&\left|\|d_A^{+,*}(\chi_0v)\|_{L^2(X)} - \|d_{A_0}^{+,*}(\chi_0v)\|_{L^2(X)}\right|
\\
&\quad \leq
\|*(a\wedge \chi_0v)\|_{L^2(X)}
\\
&\quad \leq 2\|a\|_{L^p(\supp\chi_0)} \|v\|_{L^q(X)}
\\
&\quad \leq c_p\|a\|_{L^p(U)} \left(\|d_A^+d_A^{+,*}v\|_{L^r(X)} + \|v\|_{L^r(X)}\right)
\quad\hbox{(by \eqref{eq:Feehan_5-3-1_Lp_dA+dA+*})}
\\
&\quad = c_p\|a\|_{L^p(U)} (\mu(A) + 1)\|v\|_{L^r(X)}
\quad\hbox{(by \eqref{eq:Least_eigenvalue_dA+dA+*})},
\end{align*}
where we used the fact that $\supp\chi_0 \subset U$ by construction and $c_p=c_p(g,p)$. Thus, noting that
$$
\|v\|_{L^r(X)} \leq \left(\Vol_g(X)\right)^{1/q}\|v\|_{L^2(X)},
$$
and $\|v\|_{L^2(X)}=1$, we see that
\begin{equation}
\label{eq:Upper_bound_dA*chi0v_L2norm_minus_dA0*chi0v_L2norm_for_A_Lp_loc_near_A0}
\left|\|d_A^{+,*}(\chi_0v)\|_{L^2(X)} - \|d_{A_0}^{+,*}(\chi_0v)\|_{L^2(X)}\right|
\leq
c_p\|a\|_{L^p(U)} (\mu(A) + 1),
\end{equation}
for a positive constant, $c_p=c_p(g,p)$. Next, we observe that
\begin{align*}
\|d_{A_0}^{+,*}(\chi_0v)\|_{L^2(X)}
&\geq
\sqrt{\mu(A_0)}\|\chi_0v\|_{L^2(X)} \quad\hbox{(by \eqref{eq:Least_eigenvalue_dA+dA+*})}
\\
&\geq
\sqrt{\mu(A_0)}\left( \|v\|_{L^2(X)} - \sum_{l=1}^L \|\chi_lv\|_{L^2(X)}\right)
\\
&\geq
\sqrt{\mu(A_0)}\left( \|v\|_{L^2(X)}
- \|v\|_{L^4(X)}\sum_{l=1}^L \left(\Vol_g(\supp\chi_l)\right)^{1/4} \right)
\\
&\geq \sqrt{\mu(A_0)}\left( \|v\|_{L^2(X)}
- cL\rho\|v\|_{L^4(X)}\right)
\\
&\geq \sqrt{\mu(A_0)}\|v\|_{L^2(X)}
- cL\rho\left(\|d_A^{+,*}v\|_{L^2(X)}+\|v\|_{L^2(X)}\right),
\end{align*}
where $c=c(g)$ and we used \eqref{eq:Feehan_Leness_6-6-1_L4v_L2dA+*v} to obtain the preceding inequality. Therefore, because $\|d_A^{+,*}v\|_{L^2(X)} = \sqrt{\mu(A)}\|v\|_{L^2(X)}$ by \eqref{eq:Least_eigenvalue_dA+dA+*} and $\|v\|_{L^2(X)}=1$,
\begin{equation}
\label{eq:Lower_bound_dA*chi0v_L2norm_for_A_Lp_loc_near_A0}
\|d_{A_0}^{+,*}(\chi_0v)\|_{L^2(X)}
\geq \sqrt{\mu(A_0)} - cL\rho\left(\sqrt{\mu(A)}+1\right).
\end{equation}
Observe that
\[
\|d_{A_0}^{+,*}(\chi_0v)\|_{L^2(X)}
\leq
\|d_A^{+,*}(\chi_0v)\|_{L^2(X)}
+ \left| \|d_{A_0}^{+,*}(\chi_0v)\|_{L^2(X)} - \|d_A^{+,*}(\chi_0v)\|_{L^2(X)} \right|.
\]
We rewrite the preceding inequality and combine with \eqref{eq:Upper_bound_dA*chi0v_L2norm_minus_dA0*chi0v_L2norm_for_A_Lp_loc_near_A0} and
\eqref{eq:Lower_bound_dA*chi0v_L2norm_for_A_Lp_loc_near_A0} to give
\begin{align*}
\|d_A^{+,*}(\chi_0v)\|_{L^2(X)}
&\geq
\|d_{A_0}^{+,*}(\chi_0v)\|_{L^2(X)}
- \left|\|d_A^{+,*}(\chi_0v)\|_{L^2(X)} - \|d_{A_0}^{+,*}(\chi_0v)\|_{L^2(X)}\right|
\\
&\geq \sqrt{\mu(A_0)} - cL\rho\left(\sqrt{\mu(A)}+1\right)
- c_p\|a\|_{L^p(U)}(\mu(A) + 1),
\end{align*}
for positive constants, $c=c(g)$ and $c_p=c_p(g,p)$. We substitute the preceding inequality, together with \eqref{eq:Upper_bound_dA*v_dA*chiv_L2innerproduct_for_A_Lp_loc_near_A0} and \eqref{eq:Upper_bound_dA*chiv_L2norm_for_A_Lp_loc_near_A0}, in \eqref{eq:Eigenvalue_identity_muA_for_A_Lp_loc_near_A0} to discover that $\mu(A)$ obeys
\begin{multline}
\label{eq:Lower_bound_sqrt_muA_for_A_Lp_loc_near_A0_preliminary}
\mu(A)
\geq
\left(\sqrt{\mu(A_0)} - cL\rho\left(\sqrt{\mu(A)}+1\right)
- c_p\|a\|_{L^p(U)}(\mu(A) + 1)\right)^2
\\
- cL\rho^2\mu(A)(1+\mu(A)) - cL\rho^{1/3}(\mu(A)+1)^2,
\end{multline}
a preliminary lower bound for $\mu(A)$, where $c=c(g)$ and $c_p=c_p(g,p)$.
\end{step}

\begin{step}[Upper and lower bounds for $\mu(A)$]
The inequality \eqref{eq:Lower_bound_sqrt_muA_for_A_Lp_loc_near_A0_preliminary} implies an upper bound for $\mu(A_0)$ in terms of $\mu(A)$ and hence, by interchanging the roles of $A$ and $A_0$, an upper bound for $\mu(A)$ in terms of $\mu(A_0)$. Therefore, regarding $A_0$ as fixed, for small enough $\delta = \delta(\mu(A_0),g,L,p) \in (0,1]$ and $\rho_0 = \rho_0(\mu(A_0),g,L)\in (0, 1\wedge \Inj(X,g)]$, recalling that $\rho \in (0,\rho_0]$ and $\|a\|_{L^p(U)} \leq \delta$ by hypothesis, we may suppose that
$$
\sqrt{\mu(A_0)} - cL\rho\left(\sqrt{\mu(A)}+1\right)
- c_p\|a\|_{L^p(U)}(\mu(A) + 1) \geq 0.
$$
Thus, using the elementary inequality, $(x+y)^{1/2} \leq x^{1/2} + y^{1/2}$ for $x,y \geq 0$, we obtain from \eqref{eq:Lower_bound_sqrt_muA_for_A_Lp_loc_near_A0_preliminary} that
\begin{multline*}
\sqrt{\mu(A)} + \left(cL\rho^2\mu(A)(1+\mu(A)) + cL\rho^{1/3}(\mu(A)+1)^2\right)^{1/2}
\\
\geq
\sqrt{\mu(A_0)} - cL\rho\left(\sqrt{\mu(A)}+1\right)
- c_p\|a\|_{L^p(U)}(\mu(A) + 1).
\end{multline*}
The preceding inequality yields the desired lower bound \eqref{eq:Lower_bound_sqrt_muA_for_A_Lp_loc_near_A0} for $\mu(A)$, after an another application of the elementary inequality, $(x+y)^{1/2} \leq x^{1/2} + y^{1/2}$ for $x,y \geq 0$ to give
\begin{align*}
{}&\left(cL\rho^2\mu(A)(1+\mu(A)) + cL\rho^{1/3}(\mu(A)+1)^2\right)^{1/2}
\\
&\quad \leq c\sqrt{L}\,\rho\sqrt{\mu(A)}\left(1+\sqrt{\mu(A)}\right) + c\sqrt{L}\,\rho^{1/6}(\mu(A)+1)
\\
&\quad \leq c\sqrt{L}\,\rho^{1/6}(\mu(A)+1).
\end{align*}
Interchanging the roles of $A$ and $A_0$ in the preceding inequality yields the desired upper bound \eqref{eq:Upper_bound_sqrt_muA_for_A_Lp_loc_near_A0} for $\mu(A)$.
\end{step}

This completes the proof of Proposition \ref{prop:Lp_loc_continuity_least_eigenvalue_wrt_connection}.
\end{proof}

%COMMENT Refinement is for arXiv v2
\begin{comment}
We could shorten and streamline the proofs of Propositions \ref{prop:Lp_loc_continuity_least_eigenvalue_wrt_connection} and \ref{prop:L2_loc_continuity_least_eigenvalue_wrt_connection} and slightly sharpen the results using a claim that gives an L^q or L^\infty bound for the eigenvector v near the beginning of the proofs, avoiding the need to keep referring back to earlier elliptic estimates for v.
\end{comment}

%%%%%%%%%%%%%%%%%%%%%%%%%%%%%%%%%%%%%%%%%%%%%%%%%%%%%%%%%%%%%%%%%%%%%%%%%%%%%%%
%
%                                bibliography
%
%%%%%%%%%%%%%%%%%%%%%%%%%%%%%%%%%%%%%%%%%%%%%%%%%%%%%%%%%%%%%%%%%%%%%%%%%%%%%%%

\bibliography{master,mfpde}

\def\cprime{$'$} \def\cprime{$'$}
  \def\ocirc#1{\ifmmode\setbox0=\hbox{$#1$}\dimen0=\ht0 \advance\dimen0
  by1pt\rlap{\hbox to\wd0{\hss\raise\dimen0
  \hbox{\hskip.2em$\scriptscriptstyle\circ$}\hss}}#1\else {\accent"17 #1}\fi}
  \def\cprime{$'$} \def\cprime{$'$} \def\cprime{$'$} \def\cprime{$'$}
  \def\polhk#1{\setbox0=\hbox{#1}{\ooalign{\hidewidth
  \lower1.5ex\hbox{`}\hidewidth\crcr\unhbox0}}} \def\cprime{$'$}
  \def\cprime{$'$} \def\cprime{$'$}
  \def\lfhook#1{\setbox0=\hbox{#1}{\ooalign{\hidewidth
  \lower1.5ex\hbox{'}\hidewidth\crcr\unhbox0}}} \def\cprime{$'$}
  \def\cprime{$'$} \def\cprime{$'$} \def\cprime{$'$} \def\cprime{$'$}
\providecommand{\bysame}{\leavevmode\hbox to3em{\hrulefill}\thinspace}
\providecommand{\MR}{\relax\ifhmode\unskip\space\fi MR }
\providecommand{\MRhref}[2]{%
  \href{http://www.ams.org/mathscinet-getitem?mr=#1}{#2}
}
\providecommand{\href}[2]{#2}
\begin{thebibliography}{10}

\bibitem{AdamsFournier}
R.~A. Adams and J.~J.~F. Fournier, \emph{Sobolev spaces}, second ed.,
  Elsevier/Academic Press, Amsterdam, 2003. \MR{2424078 (2009e:46025)}

\bibitem{AtiyahGeomYM}
M.~F. Atiyah, \emph{Geometry of {Y}ang-{M}ills fields}, Scuola Normale
  Superiore Pisa, Pisa, 1979. \MR{554924 (81a:81047)}

\bibitem{ADHM}
M.~F. Atiyah, N.~J. Hitchin, V.~G. Drinfel{\cprime}d, and Yu.~I. Manin,
  \emph{Construction of instantons}, Phys. Lett. A \textbf{65} (1978),
  185--187. \MR{598562 (82g:81049)}

\bibitem{AHS}
M.~F. Atiyah, N.~J. Hitchin, and I.~M. Singer, \emph{Self-duality in
  four-dimensional {R}iemannian geometry}, Proc. Roy. Soc. London Ser. A
  \textbf{362} (1978), no.~1711, 425--461. \MR{506229 (80d:53023)}

\bibitem{Aubin_1998}
T.~Aubin, \emph{Some nonlinear problems in {R}iemannian geometry}, Springer,
  Berlin, 1998. \MR{1636569 (99i:58001)}

\bibitem{Bor_1992}
G.~Bor, \emph{Yang-{M}ills fields which are not self-dual}, Comm. Math. Phys.
  \textbf{145} (1992), 393--410. \MR{1162805 (93e:58030)}

\bibitem{Bor_Montgomery_1990}
G.~Bor and R.~Montgomery, \emph{{${\rm SO}(3)$} invariant {Y}ang-{M}ills fields
  which are not self-dual}, Hamiltonian systems, transformation groups and
  spectral transform methods ({M}ontreal, {PQ}, 1989), Univ. Montr\'eal,
  Montreal, QC, 1990, pp.~191--198. \MR{1110384 (92f:58041)}

\bibitem{Bourguignon_Lawson_1981}
J-P. Bourguignon and H.~B. Lawson, Jr., \emph{Stability and isolation phenomena
  for {Y}ang-{M}ills fields}, Comm. Math. Phys. \textbf{79} (1981), 189--230.
  \MR{612248 (82g:58026)}

\bibitem{Bourguignon_Lawson_Simons_1979}
J-P. Bourguignon, H.~B. Lawson, Jr., and J.~Simons, \emph{Stability and gap
  phenomena for {Y}ang-{M}ills fields}, Proc. Nat. Acad. Sci. U.S.A.
  \textbf{76} (1979), 1550--1553. \MR{526178 (80h:53028)}

\bibitem{Buchdahl_1986}
N.~P. Buchdahl, \emph{Instantons on {${\bf C}{\rm P}_2$}}, J. Differential
  Geom. \textbf{24} (1986), 19--52. \MR{857374 (88b:32066)}

\bibitem{Dodziuk_Min-Oo_1982}
J.~Dodziuk and M.~Min-Oo, \emph{An {$L_{2}$}-isolation theorem for
  {Y}ang-{M}ills fields over complete manifolds}, Compositio Math. \textbf{47}
  (1982), 165--169. \MR{677018 (84b:53033b)}

\bibitem{DonASD}
S.~K. Donaldson, \emph{Anti self-dual {Y}ang-{M}ills connections over complex
  algebraic surfaces and stable vector bundles}, Proc. London Math. Soc. (3)
  \textbf{50} (1985), 1--26. \MR{765366 (86h:58038)}

\bibitem{Donaldson_1985geomtoday}
\bysame, \emph{Vector bundles on the flag manifold and the {W}ard
  correspondence}, Geometry today ({R}ome, 1984), Progr. Math., vol.~60,
  Birkh\"auser Boston, Boston, MA, 1985, pp.~109--119. \MR{895150 (88g:32048)}

\bibitem{DonPoly}
\bysame, \emph{Polynomial invariants for smooth four-manifolds}, Topology
  \textbf{29} (1990), 257--315.

\bibitem{DK}
S.~K. Donaldson and P.~B. Kronheimer, \emph{The geometry of four-manifolds},
  Oxford University Press, New York, 1990. \MR{1079726 (92a:57036)}

\bibitem{Feehan_yangmillsenergygapflat}
P.~M.~N. Feehan, \emph{Energy gap for {Y}ang-{M}ills connections, {II}:
  Arbitrary closed {R}iemannian manifolds}, arXiv:1502.00668.

\bibitem{Feehan_yang_mills_gradient_flow_v2}
\bysame, \emph{Global existence and convergence of smooth solutions to
  {Y}ang-{M}ills gradient flow over compact four-manifolds}, arXiv:1409.1525v2,
  xx+471 pages.

\bibitem{FeehanSlice}
\bysame, \emph{Critical-exponent {S}obolev norms and the slice theorem for the
  quotient space of connections}, Pacific J. Math. \textbf{200} (2001),
  71--118, arXiv:dg-ga/9711004.

\bibitem{FLKM1}
P.~M.~N. Feehan and T.~G. Leness, \emph{{D}onaldson invariants and
  wall-crossing formulas. {I}: Continuity of gluing and obstruction maps},
  arXiv:math/9812060.

\bibitem{FU}
D.~S. Freed and K.~K. Uhlenbeck, \emph{Instantons and four-manifolds}, second
  ed., Mathematical Sciences Research Institute Publications, vol.~1, Springer,
  New York, 1991. \MR{1081321 (91i:57019)}

\bibitem{Gerhardt_2010}
C.~Gerhardt, \emph{An energy gap for {Y}ang-{M}ills connections}, Comm. Math.
  Phys. \textbf{298} (2010), 515--522. \MR{2669447 (2012d:58019)}

\bibitem{Gritsch_2000}
U.~Gritsch, \emph{Morse theory for the {Y}ang-{M}ills functional via
  equivariant homotopy theory}, Trans. Amer. Math. Soc. \textbf{352} (2000),
  no.~8, 3473--3493. \MR{1695023 (2000m:58027)}

\bibitem{GroisserParkerSphere}
D.~Groisser and T.~H. Parker, \emph{The {R}iemannian geometry of the
  {Y}ang-{M}ills moduli space}, Comm. Math. Phys. \textbf{112} (1987),
  663--689. \MR{910586 (89b:58024)}

\bibitem{Knapp_2002}
A.~W. Knapp, \emph{Representation theory of semisimple groups}, Princeton
  Mathematical Series, vol.~36, Princeton University Press, Princeton, NJ,
  1986. \MR{855239 (87j:22022)}

\bibitem{Kozono_Maeda_Naito_1995}
H.~Kozono, Y.~Maeda, and H.~Naito, \emph{Global solution for the {Y}ang-{M}ills
  gradient flow on {$4$}-manifolds}, Nagoya Math. J. \textbf{139} (1995),
  93--128. \MR{1355271 (97a:58038)}

\bibitem{KMStructure}
P.~B. Kronheimer and T.~S. Mrowka, \emph{Embedded surfaces and the structure of
  {D}onaldson's polynomial invariants}, J. Differential Geom. \textbf{41}
  (1995), 573--734. \MR{1338483 (96e:57019)}

\bibitem{Li_1980}
P.~Li, \emph{On the {S}obolev constant and the {$p$}-spectrum of a compact
  {R}iemannian manifold}, Ann. Sci. \'Ecole Norm. Sup. (4) \textbf{13} (1980),
  451--468. \MR{608289 (82h:58054)}

\bibitem{Min-Oo_1982}
M.~Min-Oo, \emph{An {$L_{2}$}-isolation theorem for {Y}ang-{M}ills fields},
  Compositio Math. \textbf{47} (1982), 153--163. \MR{0677017 (84b:53033a)}

\bibitem{MorganMrowkaPoly}
J.~W. Morgan and T.~S. Mrowka, \emph{A note on {D}onaldson's polynomial
  invariants}, Internat. Math. Res. Notices (1992), 223--230. \MR{1191573
  (93m:57032)}

\bibitem{ParkerGauge}
T.~H. Parker, \emph{Gauge theories on four-dimensional {R}iemannian manifolds},
  Comm. Math. Phys. \textbf{85} (1982), 563--602. \MR{677998 (84b:58036)}

\bibitem{Parker_1992invent}
\bysame, \emph{Nonminimal {Y}ang-{M}ills fields and dynamics}, Invent. Math.
  \textbf{107} (1992), 397--420. \MR{1144429 (93b:58037)}

\bibitem{Rade_1992}
J.~R\r{a}de, \emph{On the {Y}ang-{M}ills heat equation in two and three
  dimensions}, J. Reine Angew. Math. \textbf{431} (1992), 123--163. \MR{1179335
  (94a:58041)}

\bibitem{Sadun_1994}
L.~Sadun, \emph{A symmetric family of {Y}ang-{M}ills fields}, Comm. Math. Phys.
  \textbf{163} (1994), no.~2, 257--291. \MR{1284785 (95f:53059)}

\bibitem{Sadun_Segert_1991}
L.~Sadun and J.~Segert, \emph{Non-self-dual {Y}ang-{M}ills connections with
  nonzero {C}hern number}, Bull. Amer. Math. Soc. (N.S.) \textbf{24} (1991),
  163--170. \MR{1067574 (91m:58038)}

\bibitem{Sadun_Segert_1992cmp}
\bysame, \emph{Non-self-dual {Y}ang-{M}ills connections with quadrupole
  symmetry}, Comm. Math. Phys. \textbf{145} (1992), 363--391. \MR{1162804
  (94b:53056)}

\bibitem{Sadun_Segert_1992cpam}
\bysame, \emph{Stationary points of the {Y}ang-{M}ills action}, Comm. Pure
  Appl. Math. \textbf{45} (1992), 461--484. \MR{1161540 (93d:58034)}

\bibitem{Sadun_Segert_1993}
\bysame, \emph{Constructing non-self-dual {Y}ang-{M}ills connections on {$S^4$}
  with arbitrary {C}hern number}, Differential geometry: geometry in
  mathematical physics and related topics ({L}os {A}ngeles, {CA}, 1990), Proc.
  Sympos. Pure Math., vol.~54, Amer. Math. Soc., Providence, RI, 1993,
  pp.~529--537. \MR{1216561}

\bibitem{Sedlacek}
S.~Sedlacek, \emph{A direct method for minimizing the {Y}ang-{M}ills functional
  over {$4$}-manifolds}, Comm. Math. Phys. \textbf{86} (1982), 515--527.
  \MR{679200 (84e:81049)}

\bibitem{Shen_1982}
C.~L. Shen, \emph{The gap phenomena of {Y}ang-{M}ills fields over the complete
  manifold}, Math. Z. \textbf{180} (1982), 69--77. \MR{656222 (83k:53048)}

\bibitem{SibnerSibnerUhlenbeck}
L.~M. Sibner, R.~J. Sibner, and K.~K. Uhlenbeck, \emph{Solutions to
  {Y}ang-{M}ills equations that are not self-dual}, Proc. Nat. Acad. Sci.
  U.S.A. \textbf{86} (1989), 8610--8613. \MR{1023811 (90j:58032)}

\bibitem{TauSelfDual}
C.~H. Taubes, \emph{Self-dual {Y}ang-{M}ills connections on non-self-dual
  {$4$}-manifolds}, J. Differential Geom. \textbf{17} (1982), 139--170.
  \MR{658473 (83i:53055)}

\bibitem{TauPath}
\bysame, \emph{Path-connected {Y}ang-{M}ills moduli spaces}, J. Differential
  Geom. \textbf{19} (1984), 337--392. \MR{755230 (85m:58049)}

\bibitem{TauIndef}
\bysame, \emph{Self-dual connections on {$4$}-manifolds with indefinite
  intersection matrix}, J. Differential Geom. \textbf{19} (1984), 517--560.
  \MR{755237 (86b:53025)}

\bibitem{TauFrame}
\bysame, \emph{A framework for {M}orse theory for the {Y}ang-{M}ills
  functional}, Invent. Math. \textbf{94} (1988), 327--402. \MR{958836
  (90a:58035)}

\bibitem{TauStable}
\bysame, \emph{The stable topology of self-dual moduli spaces}, J. Differential
  Geom. \textbf{29} (1989), 163--230. \MR{978084 (90f:58023)}

\bibitem{UhlLp}
K.~K. Uhlenbeck, \emph{Connections with {$L^{p}$} bounds on curvature}, Comm.
  Math. Phys. \textbf{83} (1982), 31--42. \MR{648356 (83e:53035)}

\bibitem{UhlRem}
\bysame, \emph{Removable singularities in {Y}ang-{M}ills fields}, Comm. Math.
  Phys. \textbf{83} (1982), 11--29. \MR{648355 (83e:53034)}

\bibitem{Xin_1984}
Y.~L. Xin, \emph{Remarks on gap phenomena for {Y}ang-{M}ills fields}, Sci.
  Sinica Ser. A \textbf{27} (1984), 936--942, arXiv:math/0203077. \MR{767626
  (86f:58041)}

\end{thebibliography}
\bibliographystyle{amsplain}

\end{document}